\documentclass[11pt]{amsart}
\usepackage[utf8]{inputenc}
\usepackage{amsmath,amssymb,amsthm, mathrsfs}
\usepackage{indentfirst}
\usepackage{comment,relsize,braket}
\usepackage{enumerate}
\usepackage[numbers]{natbib}
\usepackage{graphicx}
\usepackage{color}
\usepackage[all]{xy}
\usepackage{fullpage}
\usepackage{tikz}
\usepackage[pagebackref,colorlinks, linkcolor=blue!70!black]{hyperref}
\usepackage[top=1in, bottom=1.3in, left=.8in, right=.8in]{geometry}
\usepackage{float}
\newtheorem{Theorem}{Theorem}[section]
\newtheorem{Lemma}[Theorem]{Lemma}
\newtheorem{Corollary}[Theorem]{Corollary}
\newtheorem{Proposition}[Theorem]{Proposition}
\newtheorem{Remark}[Theorem]{Remark}

\newtheorem{Example}[Theorem]{Example}

\newtheorem{Definition}[Theorem]{Definition}

\newtheorem{Conjecture}[Theorem]{Conjecture}

\numberwithin{equation}{section}

\def\QQ{{\mathbb Q}} \def\NN{{\mathbb N}} \def\ZZ{{\mathbb Z}}

\def\Hc{{\mathcal H}} \def\Cc{{\mathcal C}}

\def\opn#1#2{\def#1{\operatorname{#2}}} 
\opn\chara{char} \opn\length{\ell} \opn\pd{pd} \opn\rk{rk}
\opn\pdim{pdim} \opn\injdim{inj\,dim}
\opn\rank{rank} \opn\depth{depth} \opn\grade{grade} 
\opn\hei{ht} \opn\embdim{emb\,dim}\opn\codim{codim}
\opn\Tr{Tr} \opn\bigrank{big\,rank}
\opn\superheight{superheight} \opn\lcm{lcm}
\opn\rdim{rdim} \opn\trdeg{tr\,deg} \opn\reg{reg}  \opn\lreg{lreg} 
\opn\ini{in} \opn\lpd{lpd} \opn\size{size} \opn{\mult}{mult}
\opn\div{div} \opn\Div{Div} \opn\cl{cl} \opn\Cl{Cl}
\opn\Spec{Spec} \opn\Supp{Supp} \opn\supp{supp} 
\opn\Sing{Sing} \opn\Ass{Ass} \opn\Min{Min}
\opn\Proj{Proj} \opn{\Max}{Max} \opn{\Assh}{Assh}
\opn\Ann{Ann} \opn\Rad{Rad} \opn\Soc{Soc}
\opn\Syz{Syz} \opn\Im{Im} \opn\Ker{Ker} \opn\Coker{Coker}
\opn\Am{Am} \opn\Hom{Hom} \opn\Tor{Tor} \opn\Ext{Ext}

\opn{\cone}{cone}
\opn{\PF}{PF}
\opn{\F}{F}
\opn{\QF}{QF}
\opn{\RF}{RF}
\opn{\D}{D}
\opn{\H}{H}
\opn{\t}{t}
\opn{\Ap}{Ap}
\opn{\UF}{UF}
\opn{\MPD}{MPD}
\opn{\Maximals}{Maximals}

\begin{document}

\title{Affine semigroups of maximal projective dimension}

\address{IIT Gandhinagar, Palaj, Gandhinagar, Gujarat-382355 India}

\author{Om Prakash Bhardwaj}
\email{om.prakash@iitgn.ac.in}

\author{Kriti Goel}
\email{kritigoel.maths@gmail.com}

\author{Indranath Sengupta}
\email{indranathsg@iitgn.ac.in} 
\thanks{The second author is supported by the Early Career Fellowship at IIT Gandhinagar.}
\thanks{The third author is the corresponding author; supported by the MATRICS research grant MTR/2018/000420, sponsored by the SERB, Government of India.}
\thanks{An extended abstract of the paper is published in the proceedings of The 34th International Conference on Formal Power Series and Algebraic Combinatorics held at Indian Institute of Science, Bangalore (India) during July 18-22, 2022.}

\maketitle

\begin{abstract}
	A submonoid of $\mathbb{N}^d$ is of maximal projective dimension ($\MPD$) if the associated affine semigroup ring has the maximum possible projective dimension. Such submonoids have a nontrivial set of pseudo-Frobenius elements. We generalize the notion of symmetric semigroups, pseudo-symmetric semigroups, and row-factorization matrices for pseudo-Frobenius elements of numerical semigroups to the case of $\MPD$-semigroups in $\mathbb{N}^d$.  
	Under suitable conditions, we prove that these semigroups satisfy the generalized Wilf’s conjecture. We prove that the generic nature of the defining ideal of the associated semigroup ring of an $\MPD$-semigroup implies uniqueness of row-factorization matrix for each pseudo-Frobenius element. Further, we give a description of pseudo-Frobenius elements and row-factorization matrices of gluing of $\MPD$-semigroups. We prove that the defining ideal of gluing of $\MPD$-semigroups is never generic.
\end{abstract}

\section{Introduction}

Let $\ZZ$ and $\NN$ denote the sets of integers and non-negative integers respectively. An affine semigroup $S$ is a finitely generated submonoid of $\NN^d$ for some positive integer $d.$ When $d=1$, affine semigroups correspond to numerical semigroups. Equivalently, a submonoid $S$ of $\NN$ is called a numerical semigroup if it has a finite complement in $\NN.$ If $S \neq \NN,$ then the largest integer not belonging to $S$ is known as the Frobenius number of $S,$ denoted by $\F(S).$ Also, the finiteness of $\NN \setminus S$ implies that there exists at least one element $f \in \NN \setminus S$ such that $f + (S \setminus \{0\}) \subset S.$ These elements are called pseudo-Frobenius elements of the numerical semigroup $S$. The set of pseudo-Frobenius elements is denoted by $\PF(S).$ In particular, $\F(S)$ is a pseudo-Frobenius number. But for affine semigroups in $\NN^d,$ the existence of such elements is not always guaranteed. The study of pseudo-Frobenius elements in affine semigroups over $\NN^d$ is done in \cite{pfelements}, where the authors consider the complement of the affine semigroup in its rational polyhedral cone. They give a necessary and a sufficient condition for the existence of pseudo-Frobenius elements using properties of the associated semigroup ring. Let $S$ be an affine semigroup and let $k$ be a field. The semigroup ring $k[S] = \bigoplus_{s \in S} k \ {\bf t}^s$ of $S$ is a $k$-subalgebra of the polynomial ring $k[t_1,\ldots,t_d],$ where $t_1,\ldots,t_d$ are indeterminates and ${\bf t}^s = \prod_{i=1}^{d} t_i^{s_i}$ for all $s = (s_1,\ldots,s_d) \in S.$ In \cite{pfelements}, the authors prove that an affine semigroup $S$ has pseudo-Frobenius elements if and only if the length of the graded minimal free resolution of the corresponding semigroup ring is maximal. Affine semigroups having pseudo-Frobenius elements are called maximal projective dimension ($\MPD$) semigroups. We call the cardinality of the set of pseudo-Frobenius elements the Betti-type of $S,$ and it is denoted by $\beta$-$\t(S).$ Note that this is not the Cohen-Macaulay type of $S$ since $\MPD$-semigroups, when $d \geq 2,$ are not Cohen-Macaulay.

One of the widely studied class of numerical semigroups is symmetric semigroups - numerical semigroups with a unique pseudo-Frobenius element.
The motivation to study these semigroups comes from the work E. Kunz, who proved that a one-dimensional analytically irreducible Noetherian local ring is Gorenstein if and only if its value semigroup is symmetric. In other words, a numerical semigroup is symmetric if and only if the associated semigroup ring is Gorenstein. Symmetric numerical semigroups have odd Frobenius number. A numerical semigroup $S$ is called pseudo-symmetric if $\F(S)$ is even and $\PF(S) = \{ \F(S)/2, \F(S) \}.$ We study a generalization of these notions to the case of $\MPD$-semigroups. Let $S$ be an $\MPD$-semigroup and let $\cone(S)$ denote the rational polyhedral cone of $S.$ Set $\Hc(S) = (\cone(S) \setminus S) \cap \NN^d.$ For a fixed term order $\prec$ on $\NN^d$, we define the Frobenius element as $\F(S)_{\substack{\\\prec}} = \max_{\substack{\\\prec}} \Hc(S).$ Note that in the case of $\MPD$-semigroups, Frobenius elements may not always exist, with respect to any term order. But if there is a term order $\prec$ such that $\F(S)_{\substack{\\\prec}}$ exists and $|\PF(S)|=1,$ then we say that $S$ is a $\prec$-symmetric semigroup. Further, if $\PF(S) = \{ \F(S)_{\substack{\\\prec}}/2, \F(S)_{\substack{\\\prec}} \},$ then we say that $S$ is a $\prec$-pseudo-symmetric semigroup. 

Any affine semigroup $S$ has a unique minimal generating set whose cardinality is known as the embedding dimension of $S$, and it is denoted by $e(S).$ In 1978, Wilf proposed a conjecture related to the Diophantine Frobenius Problem that claims that the inequality 
\[ \F(S)+1 \leq e(S) \cdot |\{ s \in S \mid s < \F(S) \}| \] 
is true for every numerical semigroup. While this conjecture still remains open, a potential extension of Wilf's conjecture to affine semigroups is studied in \cite{Wilfconjecture}. We prove that $\prec$-symmetric and $\prec$-pseudo-symmetric semigroups satisfy the generalized Wilf's conjecture under suitable assumptions.

In \cite{moscariello}, A. Moscariello introduced the notion of row-factorization ($\RF$) matrices associated to the pseudo-Frobenius elements of a numerical semigroup. He used this object to investigate the type of almost symmetric numerical semigroups of embedding dimension four and prove a conjecture given by T. Numata in \cite{numata}, which states that the type of an almost symmetric semigroup of embedding dimension four is at most three. In recent years, $\RF$-matrices have been studied by K. Eto, J. Herzog, and K.-i. Watanabe in (\cite{etoGeneric}, \cite{etoRowFactor}, \cite{eto2017}, \cite{herzogWatanabe}). We extend the definition of $\RF$-matrices of the pseudo-Frobenius elements to the setting of $\MPD$-semigroups. We also give a description of pseudo-Frobenius elements and their $\RF$-matrices in a glued $\MPD$-semigroup. For an affine semigroup $S,$ let $G(S)$ denote the group generated by $S$ in $\ZZ^d.$ Recall that an affine semigroup $S$ is said to be a gluing if there exists a non-trivial partition of its minimal generating set, $A_1 \amalg A_2,$ and $d \in \langle A_1 \rangle \cap \langle A_2 \rangle$ such that $G(\langle A_1 \rangle) \cap G(\langle A_2 \rangle) = d \ZZ.$  

Let $k$ be a field and $S = \langle a_1, \ldots, a_n \rangle$ be a finitely generated submonoid of $\mathbb{N}^d.$ Then the semigroup ring $k[S] = k[{\bf t}^{a_1}, \ldots, {\bf t}^{a_n}]$ of $S$ can be represented as a quotient
of a polynomial ring using a canonical surjection $\pi : k[x_1,\ldots,x_n] \rightarrow k[S]$ given by $\pi(x_i) = {\bf t}^{a_i}$ for all $i=1,\ldots,n.$ The kernel of this $k$-algebra homomorphism $\pi$, denoted by $I_S$, is a toric ideal, called defining ideal of $k[S],$ and the ring $k[S]$ is called a toric ring. A toric ideal is called generic if it has a minimal generating set consisting of binomials of full support. The notion of genericity of lattice ideal is introduced by I. Peeva and B. Sturmfels in \cite{peeva-sturmfels}. The authors give a minimal free resolution, namely the scarf complex, for the toric ring when the defining toric ideal is generic. The generic toric ideals are further studied in \cite{peeva-gasharov}, where the authors prove that  if $I_S$ is generic toric ideal, then it has a unique minimal system of generators of indispensable binomials up to the sign of binomials. In \cite{etoGeneric}, K. Eto gave a necessary and sufficient condition for the defining ideal $I_S$ of a numerical semigroup ring $k[S]$ to be generic. For $\MPD$-semigroups, we give a necessary condition for the generic nature of the defining ideal using $\RF$-matrices.

Now we summarize the contents of the paper. Unless and otherwise stated, $S$ denotes an $\MPD$-semigroup in $\NN^d,$ minimally generated by elements $a_1,\ldots,a_n \in \NN^d.$ Section 2 begins with recollection of some basic definitions, including the definition of pseudo-Frobenius elements of affine semigroups over $\mathbb{N}^d.$ In \cite[Theorem 17]{pfelements}, the authors prove that $\MPD$-semigroups are stable with respect to gluing. In Theorem \ref{PFGluing}, we give a description of the set of pseudo-Frobenius elements of a glued semigroup. In \cite{jafari}, the authors introduce the concept of quasi-Frobenius elements, $\QF(S)$, which a generalization of pseudo-Frobenius elements to the context of simplicial affine semigroups. While it is known that $\PF(S) \cap \QF(S) = \emptyset$ in any simplicial $\MPD$-semigroup $S$, we show, via an example, that the cardinality of the two sets may not be equal as well. In other words, unlike Cohen-Macaulay simplicial affine semigroups, the cardinality of $\QF(S)$ may not be equal to the last Betti number.  
It is evident that the complement of $S$ in $\cone(S)$ may not be finite. Set $\Hc(S) = (\cone(S) \setminus S ) \cap \mathbb{N}^d$ and define a relation $\leq$ on $\Hc(S)$ as follows: ${\bf x} \leq {\bf y}$ if ${\bf y} - {\bf x} \in S.$ We show that while $\PF(S)$ may not be the set of maximal elements of $\Hc(S)$ with respect to the order $\leq,$ but if $\Hc(S)$ is finite, then $\PF(S) = \text{Maximals}_{\leq} \Hc(S).$ In section 3, we define $\prec$-symmetric, $\prec$-almost symmetric and $\prec$-pseudo-symmetric ($\MPD$)-semigroups with respect to a term order $\prec.$ The definition of these semigroups is dependent on the existence of the Frobenius element (with respect to some order $\prec$). In Example \ref{FrobNotexist}, we give an example of a class of $\MPD$-semigroups in which $\F(S)_{\substack{\\\prec}}$ does not exist for any term order $\prec.$ When $\Hc(S)$ is finite, $S$ is known as $\Cc$-semigroup. In Theorem \ref{precsymmetricity} and \ref{precPseudosymmetricity}, we give a characterization of $\prec$-symmetric and $\prec$-pseudo-symmetric $\Cc$-semigroups. A different characterization of these semigroups is also given when $S$ is a $\Cc$-semigroup and $\cone(S) \cap \NN^d = \NN^d.$ Using this characterization, in Theorem \ref{Wilf}, we prove that the semigroups in the above setup satisfy the extended Wilf's conjecture.

In section 4, we explore the connection between the Hilbert series of the semigroup ring $k[S]$ and the set of pseudo-Frobenius elements of the $\MPD$-semigroup $S.$ In Theorem \ref{FrobExp}, we prove that if $S$ is a $\Cc$-semigroup and $\cone(S) \cap \NN^d = \NN^d,$ then $F(S)_{\prec}$ is the difference of the exponent of leading term of the numerator of the Hilbert series of $k[S]$ and the generators of the semigroup $S.$  

In section 5, we introduce and study $\RF$-matrices of pseudo-Frobenius elements in $\MPD$-semigroups. In Theorem \ref{detRF}, we prove that the determinant of $\RF$-matrices is zero unless the group generated by $S$ has rank $1.$ The difference of any two rows of an $\RF$-matrix produces a binomial in the defining ideal $I_S$, called an $\RF$-relation. In \cite{herzogWatanabe}, Herzog and Watanabe raised a question - For a numerical semigroup $S$, when is $I_S$ generated by $\RF$-relations? We ask the same question for $\MPD$-semigroups and in Theorem \ref{RFRelation}, we prove that a sufficient condition for the defining ideal to be generated by $\RF$-relations, known for numerical semigroups, can be generalized for $\MPD$-semigroups. Using the description of pseudo-Frobenius elements of gluing of two $\MPD$-semigroups, we obtain a way to write the $\RF$-matrices for the pseudo-Frobenius elements of gluing. It is then natural to ask the following question - Is the property of the defining ideal being generated by $\RF$-relations compatible with gluing? In Theorem \ref{GenRFRelations}, we prove that under suitable assumptions, the defining ideal of a gluing is generated by $\RF$-relations if the initial defining ideals are so.

Section 6 explores another connection between the $\RF$-matrices of pseudo-Frobenius elements and the defining ideal of the semigroup ring. In Theorem \ref{genericRFunique}, we prove that if the defining ideal of the semigroup ring of an $\MPD$-semigroup is generic, then the $\RF$-matrix is unique for each pseudo-Frobenius element. We end the section by proving that the defining ideal of gluing of $\MPD$-semigroups is never generic.

Lastly, the authors would like to point out that while we work in the `negative space' $\Hc(S),$ it seems that most of the theory also works when we restrict to the space $((\cone(S) \cap G(S)) \setminus S) \cap \NN^d.$ The utility of this smaller space is not yet clear to us and we invite the readers to explore this possibility further.

\section{Pseudo-Frobenius elements}

Let $S$ be a finitely generated submonoid of $\NN^d$, say minimally generated by $a_1, \ldots, a_n \subseteq \NN^d.$ Such submonoids are called affine semigroups. Consider the cone of $S$ in $\QQ^d_{\geq 0}$,
\[  \cone(S) := \left\lbrace \sum_{i=1}^n \lambda_i a_i  \mid \lambda_i \in \QQ_{\geq 0}, i = 1, \ldots, n \right\rbrace \] 
and set $\Hc(S) := (\cone(S) \setminus S) \cap \NN^d.$

\begin{Definition} \label{pf}
	An element $f \in \Hc(S)$ is called a pseudo-Frobenius element of $S$ if
	$f + s \in S$ for all $s \in S \setminus \{0\}.$ The set of pseudo-Frobenius elements of $S$ is denoted by $\PF(S).$ In particular,
	\[ \PF(S) = \{ f \in \Hc(S) \mid f + a_j \in S, \ \forall j \in [1,n] \}.  \]  
\end{Definition}

Observe that the set $\PF(S)$ may be empty. Therefore, in this article, we consider such class of rings where $\PF(S)$ is non-empty. 

Let $k$ be a field. The semigroup ring $k[S]$ of $S$ is a $k$-subalgebra of the polynomial ring $k[t_1,\ldots,t_d].$ In other words, $k[S] = k[{\bf t}^{a_1}, \ldots, {\bf t}^{a_n}]$, where ${\bf t}^{a_i} = t_1^{a_{i1}} \cdots t_d^{a_{id}}$ for $a_i = (a_{i1},\ldots,a_{id})$ and for all $i=1,\ldots,n.$ Set $R = k[x_1,\ldots,x_n]$ and define a map $\pi : R \rightarrow k[S]$ given by $\pi(x_i) = {\bf t}^{a_i}$ for all $i=1,\ldots,n.$ Set $\deg x_i = a_i$ for all $i=1,\ldots,n.$ Observe that $R$ is a multi-graded ring and that $\pi$ is a degree preserving surjective $k$-algebra homomorphism. We denote by $I_S$ the kernel of $\pi.$ Then $I_S$ is a homogeneous ideal, generated by binomials, called the defining ideal of $S$. Note that a binomial $\phi = \prod_{i=1}^n x_i^{\alpha_i} - \prod_{j=1}^n x_j^{\beta_j} \in I_S$ if and only if $\sum_{i=1}^n \alpha_i a_i = \sum_{j=1}^n \beta_j a_j.$ With respect to this grading, $\deg \phi = \sum_{i=1}^n \alpha_i a_i.$

\begin{Definition}
	We say that $S$ satisfies the {\bf maximal projective dimension ($\MPD$)} if $\pdim_R k[S] = n-1.$ Equivalently, $\depth_R k[S]=1.$
\end{Definition}

In \cite[Theorem 6]{pfelements}, the authors proved that $S$ is a $\MPD$-semigroup if and only if $\PF(S) \neq \emptyset$. In particular, they prove that if $S$ is a $\MPD$-semigroup then $b \in S$ is the $S$-degree of the $(n-2)$th minimal syzygy of $k[S]$ if and only if $ b \in \left\lbrace a + \sum_{i=1}^n a_i  \mid a \in \PF(S)\right\rbrace$. Moreover, $\PF(S)$ has finite cardinality.

\begin{Example}\label{examplePF(S)}
	Let $S = \langle (2,11), (3,0), (5,9), (7,4) \rangle$. The minimal free resolution of $k[S],$ as a module over $R = k[x_1,\ldots,x_4],$ is 
	\[ 0 \rightarrow R(-(81,93)) \oplus R(-(94,82)) \rightarrow R^6 \rightarrow R^5 \rightarrow R \rightarrow k[S] \rightarrow 0. \]
	In particular, the degrees of minimal generators of the third syzygy modules are $(81,93), (94,82)$. Therefore, $S$ has two pseudo-Frobenius elements $(64,69)$ and $(77,58)$.
\end{Example}

\begin{Example}
	Let $S = \langle a_1,\ldots,a_n \rangle ~ \subseteq \mathbb{N}^d$, where ${a_i}'s$ are linearly independent over $\mathbb{Q}$. Then $k[S]$ is isomorphic to a polynomial ring in $n$ variables over the field $k$. Hence, $S$ is not $\MPD$-semigroup as $\depth_R k[S]=n.$
\end{Example}

\begin{Definition}[{\cite[Theorem 1.4]{rosales97}}]
	Let $S \subseteq \mathbb{N}^d$ be an affine semigroup and $G(S)$ be the group spanned by $S$, that is, $G(S) =\{ a-b \in \mathbb{Z}^d \mid a, b \in S\}$. Let $A$ be the minimal generating system of $S$ and $A = A_1 \amalg A_2$ be a nontrivial partition of $A$ . Let $S_i$ be the submonoid of $\mathbb{N}^d$ generated by $A_i, i \in {1, 2}$. Then $S = S_1 + S_2.$ We say that $S$ is the gluing of $S_1$ and $S_2$ by $d$ if 
	\begin{enumerate}[{\rm(1)}]
		\item $d \in S_1 \cap S_2$ and,
		\item $G(S_1 ) \cap G(S_2) = d\mathbb{Z}$. 
	\end{enumerate}
	If $S$ is a gluing of $S_1$ and $S_2$ by $d,$ we write $S = S_1 +_d S_2$.
\end{Definition}

\begin{Theorem} \label{PFGluing}
	Let $S$ be an affine semigroup such that $S = S_1 +_d S_2,$ where $S_1$ and $S_2$ are $\MPD$-semigroups. Then $S$ is a $\MPD$-semigroup and
	\begin{align*}
		\PF(S)=\{f+g+d \mid f \in \PF(S_1), \ g\in \PF(S_2) \}.
	\end{align*} 
\end{Theorem}

\begin{proof}
	The affine semigroup $S$ is a $\MPD$-semigroup follows from \cite[Theorem 17]{pfelements}.
	Set $T = \{f+g+d \mid f \in \PF(S_1), g\in \PF(S_2) \}$. By \cite[Theorem 17]{pfelements}, we have $T \subseteq \PF(S).$ Also, from \cite[Theorem 6.1]{hema2019} we have $| \PF(S) | = \beta$-$\t(S)= \beta$-$\t(S_1) \cdot \beta$-$\t(S_2) = | \PF(S_1) | \cdot | \PF(S_2) |.$ Therefore, it is sufficient to show that if $f + g + d , f'+ g' + d \in T$ such that $f + g + d = f'+ g' + d,$ then $f=f'$ and $g=g'.$ Let $f + g + d = f'+ g' + d.$ Then
	\[  (f+d) - (f'+d) = f - f' = g' - g = (g'+d)-(g+d).\] 
	Since $d \in S_1 \cap S_2,$ we get $f - f' = g'-g \in G(S_1) \cap G(S_2).$ Hence there exist a $k \in \mathbb{Z}$ such that  $f - f' = g'-g = kd.$ If $k < 0,$ then $f' = f-kd \in S_1,$ which is a contradiction. Now suppose $k > 0.$ Then $g' = g + kd \in S_2,$ which is again a contradiction. Hence $k=0$ implying that $f - f' = g'-g = 0.$
\end{proof}

Recall that a gluing two numerical semigroups is defined as follows:
\begin{Definition}  \label{NSgluing}
	Let $H_1$ and $H_2$ be two numerical semigroups minimally generated by $n_1,\ldots,n_r$ and $n_{r+1},\ldots,n_e$ respectively. Let $\lambda \in H_1 \setminus \{ n_1,\ldots,n_r \}$ and $\mu \in H_2 \setminus \{ n_{r+1}, \ldots, n_e \}$ be such that $\gcd(\lambda,\mu)=1.$ We say that $S = \langle \mu n_1, \ldots, \mu n_r, \lambda n_{r+1}, \ldots, \lambda n_{e} \rangle$ is a gluing of $H_1$ and $H_2.$
\end{Definition}

In other words, if $S$ is as in the definition above, then $S = S_1 +_{\mu\lambda} S_2,$ where $S_1 = \langle \mu n_1, \ldots, \mu n_r \rangle$ and $S_2 = \langle \lambda n_{r+1}, \ldots, \lambda n_{e} \rangle.$ Since $\PF(S_1) = \mu \PF(H_1)$ and $\PF(S_2) = \lambda \PF(H_2)$, the following result now follows as a corollary to Theorem \ref{PFGluing}.

\begin{Corollary}[{\cite[Proposition 6.6]{nari}}]
	If $S = \langle \mu H_1, \lambda H_2 \rangle$ (as in Definition {\rm \ref{NSgluing}}), then
	\[ \PF(S) = \{ \mu f + \lambda g + \mu \lambda \mid f \in \PF(H_1), g \in \PF(H_2) \}. \]
\end{Corollary}

Recall that a term order (also known as monomial ordering) on $\NN^d$ is a total order compatible with the addition of $\NN^d.$

\begin{Definition}
	Let $\prec$ be a term order on $\NN^d.$ Then $\F(S)_{\substack{\\\prec}} = \max_{\substack{\\\prec}} \Hc(S)$, if it exists, is called the Frobenius element of $S$ with respect to the term order $\prec.$ In particular, 
	\[ \F(S) = \{ \F(S)_{\substack{\\\prec}} = \max_{\prec} \Hc(S) \mid \; \prec \text{ is a term order } \}. \]
\end{Definition}
	
We write $\F(S)$ for the set of Frobenius elements of $S.$ Note that Frobenius elements may not exist. However, if $| \Hc(S) | < \infty$, then Frobenius elements do exist. Also, from \cite[Lemma 12]{pfelements}, we have that every Frobenius element is a pseudo-Frobenius element, i.e., $\F(S) \subseteq \PF(S).$

Another generalization of pseudo-Frobenius elements to the context of simplicial affine semigroups is introduced and studied in \citep{jafari}. The authors define quasi-Frobenius elements over a simplicial affine semigroup $S$ and show that the cardinality of the set of quasi-Frobenius elements equals the Cohen-Macaulay type of the semigroup ring when $k[S]$ is Cohen-Macaulay. This result also motivates the authors to call the cardinality of this set as the type of $S.$ Let $S = \langle a_1,\ldots,a_d, a_{d+1},\ldots,a_n \rangle \subseteq \NN^d$ be a simplicial affine semigroup and let $a_1,\ldots,a_d$ be situated on each extremal ray of $\cone(S).$ Let $\preceq_S$ denote the partial order on $\NN^d$ where for all elements $x,y \in \NN^d$, $x \preceq_S y$ if $y-x \in S.$ Then the set of quasi-Frobenius elements,
\[ \QF(S) = \left\lbrace w - \sum_{i=1}^{d} a_i \mid w \in \max_{\preceq_S} \cap_{i=1}^{d} \Ap(S,a_i) \right\rbrace. \]
From \cite[Remark 3.2]{jafari}, we know that in simplicial $\MPD$-semigroups, $\PF(S) \cap \QF(S) = \emptyset.$ In the following example, we show that the cardinality of the two sets may not be equal as well. In particular, type$(S)$ may not be always be equal to $\beta$-$\t(S).$
\begin{Example}{\rm 
	Let $h \in \NN$ such that $h \geq 2.$ Put $n_1 = (2h-1)2h$, $n_2 =(2h-1)(2h+1)$, $n_3=2h(2h+1),$ and $n_4=2h(2h+1)+2h-1$ and consider the semigroup
	\[ S = \langle (0,n_4),(n_1,n_4-n_1),(n_2,n_4-n_2),(n_3,n_4-n_3), (n_4,0) \rangle. \]
	($S$ is the semigroup corresponding to the projective closure of the Bresinsky curve, introduced by H. Bresinsky in {\rm\cite{bresinsky1,bresinsky2}}).
	From {\rm\cite[Lemma 6.3]{pranjal-sengupta}}, we know that 
	\begin{align*}
		I_S = \bigg\langle 
		\begin{array}{c}
		x_3x_4 - x_2x_5, \ x_3^{2h} - x_1x_4^{2h-1}, \ x_4^{4h} - x_3^{2h-1}x_5^{2h+1}, \\[1mm] 
		x_2^jx_4^{2h-j+1} - x_1x_3^{j-1}x_5^{2h-j+1}, \ x_2^{j+1}x_3^{2h-j} - x_1^2x_4^{j-1}x_5^{2h-j} \mid 1 \leq j \leq 2h 
		\end{array}
		\bigg\rangle.
	\end{align*}
	Consequently, 
	\[ I_S + \langle x_1, x_5 \rangle = \langle x_1, \ x_5, \ x_2^{2h+1}, \ x_3^{2h}, \ x_4^{4h}, \ x_3x_4, \ x_2^jx_4^{2h-j+1}, \ x_2^{j+1}x_3^{2h-j} \mid 1 \leq j \leq 2h \rangle. \]
	Hence, the image of 
	\[ \mathscr{B} = \left\lbrace 
	\begin{array}{c}
	\bigcup_{i=1}^{2h-1} \{x_2^jx_3^i, \ x_2^jx_4^i \mid 1 \leq j \leq 2h-i \} \cup \{x_2^j \mid 0 \leq j \leq 2h \} \\[1mm]
	\cup \{x_3^j \mid 1 \leq j \leq 2h-1 \} \cup \{x_4^j \mid 1 \leq j \leq 4h-1 \} 
	\end{array}
	\right\rbrace \]
	gives a $k$-basis of $k[S]/\langle x_1,x_5 \rangle = k[x_1,\ldots,x_5]/(I_S+\langle x_1,x_5 \rangle).$ Since $\Ap(S,(0,n_4)) \cap \Ap(S,(n_4,0)) = \{ \deg u \mid u \in \mathscr{B} \},$ we get
	\begin{align*}
		&\Ap(S,(0,n_4)) \cap \Ap(S,(n_4,0)) \\
		&= \left\lbrace
		\begin{array}{c}
			0, \ \bigcup_{i=1}^{2h-1} \{ j(n_1, n_4-n_1) + i(n_2, n_4-n_2), \ j(n_1, n_4-n_1) + i(n_3, n_4-n_3) \mid 1 \leq j \leq 2h-i \} \\[1mm]
			\cup \{a (n_1,n_4-n_1), \ b(n_2,n_4-n_2), \ c(n_3,n_4-n_3) \mid 1 \leq a \leq 2h, \ 1 \leq b \leq 2h-1, \ 1 \leq c \leq 4h-1 \} 
		\end{array}
		\right\rbrace.
	\end{align*} 
	Therefore,
	\begin{align*}
		&\max_{\preceq_S} \big( \Ap(S,(0,n_4)) \cap \Ap(S,(n_4,0)) \big)  \\
		&= \left\lbrace
		\begin{array}{c}
			\{(2h-i)(n_1,n_4-n_1) + i(n_2,n_4-n_2), \ (2h-i)(n_1,n_4-n_1) + i (n_3, n_4-n_3) \mid 1 \leq i \leq 2h-1 \}  \\[1mm]
			\cup \{ 2h(n_1,n_4-n_1), (4h-1)(n_3,n_4-n_3) \}
		\end{array} 
		\right\rbrace 
	\end{align*}
	implying that $|\QF(S)| = |\max_{\preceq_S} \big( \Ap(S,(0,n_4)) \cap \Ap(S,(n_4,0)) \big)| = 4h.$ Whereas, $|\PF(S)| = \beta$-$\t(S)=1$ (see {\rm\cite[Theorem 6.4]{pranjal-sengupta}}).
}\end{Example}

On $\Hc(S),$ we define a relation : ${\bf x} \leq {\bf y} $ if ${\bf y} - {\bf x} \in S.$ It is a partial order (reflexive, anti-symmetric and transitive) on $\Hc(S)$. In \cite[Proposition 2.19]{numerical}, the authors proved that if $S$ is a numerical semigroup, then $\PF(S) = \text{Maximals}_{\leq} (\ZZ \setminus S).$ In other words, $x \in \ZZ \setminus S$ if and only if $f - x \in S$ for some $f \in \PF(S).$ However, in the case of $\MPD$-semigroups over $\NN^d,$ $d \geq 2,$ we observe that this result is not true.
\begin{Example}\label{f-xnotinS}
	Let $S$ be the semigroup generated by the columns of the following matrix
	\begin{align*} 
		\left(
		\begin{tabular}{cccccccccccc}
			18 & 18 & 4 & 20 & 23 & 8 & 11 & 11 & 10 & 14 & 7 & 7 \\
			9 & 3 & 1 & 8 & 10 & 3 & 5 & 2 & 3 & 3 & 2 & 3 
		\end{tabular}
		\right).
	\end{align*}
Using Macaulay2, we get that 
\begin{align*}
0 \longrightarrow R(-(164,56)) \longrightarrow R^{18} \longrightarrow R^{120} \longrightarrow R^{441} \longrightarrow R^{1037} \longrightarrow R^{1671} \longrightarrow R^{1898} \longrightarrow R^{1519} \\
\longrightarrow R^{831} \longrightarrow R^{287} \longrightarrow R^{50} \longrightarrow R \longrightarrow k[S] \longrightarrow 0
\end{align*}
 is a minimal graded free resolution of $k[S]$ over $R$, where $R = k[x_1, \ldots, x_{12}]$.	
	Therefore, $S$ is a $\MPD$-semigroup and the degree of minimal generator of the  last syzygy module is $(164,56)$. Hence $S$ has only one pseudo-Frobenius element, which is: $(164,56) - \sum_{i=1}^{12} C_i^{T}$, where $C_i^{T}$ denote the transpose of the $i^{\mathrm{th}}$ column of the given matrix. Thus, $\PF(S) = \{ (13,4) \}$ (also see \cite[Example 5]{pfelements}). Observe that $(15,7) \in \cone(S) \setminus S$ but $(13,4) - (15,7)  = (-2,-3) \notin S.$ 
\end{Example}

This anomaly arises as $\Hc(S)$ may not always be a finite set for a $\MPD$-semigroup $S$ and hence the existence of maximal elements is not always guaranteed. However, we can still talk about maximal finite chains in $\Hc(S).$ Let $\mathscr{C}$ denote the set of all maximal finite chains (totally ordered subsets) in $\Hc(S).$ Let $c_f \in \mathscr{C}$ denote the chain with $f$ as its maximum element.

\begin{Proposition}
	Let $S$ be a $\MPD$-semigroup. Then the set map $\psi : \mathscr{C} \rightarrow \PF(S)$ defined by $\psi(c_f) = f$ is a surjective map. 
\end{Proposition} 

\begin{proof}
	Let $c_f \in \mathscr{C}.$ Suppose $\psi(c_f) = f \notin \PF(S).$ Then there exists $s \in S$ such that $f+s \notin S.$ In particular, $f+s \in \Hc(S),$ which is a contradiction to the choice of $f.$ Therefore, $\psi(c_f) \in \PF(S).$ In order to prove that $\psi$ is a surjective map, let $f \in \PF(S).$ Since $(\Hc(S), \leq)$ is a poset, there is a finite chain $c$ in $\Hc(S)$ which contains $f.$ We claim that $f$ is the maximum element of the chain $c.$ Suppose not. Then there exists an $f' \in \Hc(S),$ $f' \neq f$ such that $f \leq f'.$ It would imply $f'-f \in S$ and so there exists $s \in S$ such that $f'- f = s.$ This implies $f' = f+s \in S$, which is a contradiction.
\end{proof}

\begin{Example}
Let $S$ be as in Example \ref{f-xnotinS} (also see the figure given in \cite[Example 5]{pfelements}). Some finite maximal chains in $\Hc(S)$ with $(13,4)$ as its maximal element are
\[
\begin{array}{ll}
	(2,1) \leq (6,2) \leq (13,4), & \quad (3,1) \leq (13,4), \\
	(2,1) \leq (9,3) \leq (13,4), & \quad (5,1) \leq (13,4), \\
	(5,2) \leq (9,3) \leq (13,4), & \quad (6,1) \leq (13,4).
\end{array}
\]
\end{Example}

\begin{Remark}
	For  $f \in \PF(S),$ let $ \psi^{-1}(f) = [c_f].$ Then for any $x \in \Hc(S),$
	\[ x \leq f  \; \iff \;  x \text{ is an element of a chain in } [c_f]. \]
\end{Remark}

Observe that if $\Hc(S)$ is a finite set and $\Hc(S) \neq \emptyset$, then $\PF(S) \neq \emptyset.$ In particular, if $\Hc(S)$ is finite, then the following result holds.

\begin{Corollary}  \label{MaxPseudoFrob}
	Let $\Hc(S)$ be a non-empty finite set. Then
	\begin{enumerate}[{\rm(i)}]
		\item $\PF(S) = \Maximals_{\leq} \ \Hc(S).$
		\item Let ${\bf x} \in \NN^d.$ Then ${\bf x} \in \Hc(S)$ if and only if $f-{\bf x} \in S$ for some $f \in \PF(S).$
	\end{enumerate}
\end{Corollary}

\section{$\prec$-symmetric semigroups}

Let $S$ be a $\MPD$-semigroup. For a term order $\prec$ on $\NN^d,$ we define $\prec$-symmetric, $\prec$-almost symmetric and $\prec$-pseudo-symmetric semigroups. Note that our definitions are dependent on the existence of the Frobenius element with respect to the order $\prec.$

\begin{Definition}
	In the case $\F(S) \neq \emptyset$, we fix a term order $\prec$ such that $\F(S)_{\substack{\\\prec}} = \max_{\substack{\\\prec}} \Hc(S) \in \F(S).$
	\begin{enumerate}[{\rm (1)}]
		\item If $|\PF(S)|=1$ and $\PF(S) = \{\F(S)_{\substack{\\\prec}}\}$, then $S$ is called a { \bf $\prec$-symmetric semigroup}. In particular, if $\beta$-$\t(S)=1$ and Frobenius element exists, then $S$ is a  $\prec$-symmetric semigroup.
		
		\item Put $\PF_{\substack{\\\prec}}'(S) = \PF(S) \setminus \{ \F(S)_{\substack{\\\prec}} \}.$ If $\PF_{\substack{\\\prec}}'(S) \neq \emptyset$ and if for any $g \in \PF_{\substack{\\\prec}}'(S)$, $\F(S)_{\substack{\\\prec}} - g \in \PF_{\substack{\\\prec}}'(S)$, we say that $S$ is {\bf $\prec$-almost symmetric}. Further, if $\beta$-$\t(S)=2$, then $S$ is called {\bf $\prec$-pseudo-symmetric}. In this case, $\PF(S) = \{ \F(S)_{\substack{\\\prec}}, \F(S)_{\substack{\\\prec}}/2 \}.$ 
	\end{enumerate}
\end{Definition}

Observe that when $S$ is $\prec$-pseudo-symmetric, for some term order $\prec,$ then $\F(S)_{\substack{\\\prec}}$ has all even coordinates. Hereafter $\F(S)_{\substack{\\\prec}}$ is even means that $\F(S)_{\substack{\\\prec}}$ has all even coordinates. If $\beta$-$\t(S)=1$ and $\Hc(S)$ is a finite set then $S$ is a symmetric semigroup with respect to any term order $\prec.$ 

\begin{Example}\label{symmetricexample}
	The semigroup $S_1 = \langle (0,1),(3,0),(5,0),(1,3),(2,3) \rangle$ is a $\prec$-symmetric semigroup as $\PF(S_1) = \{(7,2)\}$ and $(7,2) = \max_{\prec} \Hc(S_1)$, where $\prec$ is a graded lexicographic order.
	Similarly, one may check that $S_2 = \langle (0,1),(2,0),(3,0),(1,3) \rangle$ is a symmetric semigroup with $\F(S_2)_{\substack{\\\prec}} = (1,2).$ \\
	The semigroup $S = \langle (0,1),(3,0),(4,0),(1,4),(5,0),(2,7) \rangle$ is $\prec$-almost symmetric semigroup with $\beta$-type $ = 2$ as $\PF(S) = \{(1,3),(2,6)\}$ and $(2,6) = \max_{\prec} \Hc(S)$, where $\prec$ is a graded lexicographic order. In particular, $S$ is a $\prec$-pseudo-symmetric semigroup.
\end{Example}

Note that in the case of affine semigroups, existence of Frobenius element cannot be guaranteed, even when $\beta$-type of the ring is one. The following example presents a class of rings where the Frobenius element does not exist.

\begin{Example}  \label{FrobNotexist}
	{\rm 
		Let $\bar{S}$ denote the semigroup corresponding to the projective closure of the Bresinsky curves. Then 
		\[ \bar{S} = \langle (0,n_4),(n_1,n_4-n_1),(n_2,n_4-n_2),(n_3,n_4-n_3), (n_4,0) \rangle,\]
		where, for $h \geq 2,$ $n_1 = (2h-1)2h$, $n_2 =(2h-1)(2h+1)$, $n_3=2h(2h+1),$ and $n_4=2h(2h+1)+2h-1.$ Let $R=k[x_1,\ldots,x_5].$ Using \cite[Theorem 6.4]{pranjal-sengupta}, we have a graded minimal free resolution of $I_{\bar{S}}$,
		\begin{center}
			$0 \rightarrow R \rightarrow R^{4h+3} \rightarrow R^{8h+4} \rightarrow R^{4h+3} \rightarrow R \rightarrow R/I_{\bar{S}} \rightarrow 0$,
		\end{center}
		implying that $\pdim_Rk[\bar{S}] = 4$ and hence $\depth_R k[\bar{S}]=1.$ Therefore, $\bar{S}$ is a $\MPD$-semigroup. Moreover, $\beta$-$\t(S)=1.$ Using the computation of the syzygies in \cite[Notation 6.2, Theorem 6.4]{pranjal-sengupta}, one obtains the degree of the last syzygy, namely, $d = (16h^3+16h^2-2h-1, 12h^2+10h-2)$. From the computation of the pseudo-Frobenius element from the degree of the last syzygy in a $\MPD$-semigroup (see \cite[Theorem 6]{pfelements}), we obtain $\PF(\bar{S}) = \{f=(16h^3-6h+1, 8h^2-6h+1)\}.$ We prove that $f$ is not a Frobenius element of $\bar{S}$ with respect to any term order $\prec$. 
		
		First, note that the group generated by $\bar{S},$ $G(\bar{S}) = \{(a,b) \in \ZZ^2 \mid a+b \in n_4 \ZZ \}$ (see \cite[Lemma 4.1]{cavaliereNiesi}) and $\cone(\bar{S}) \cap \NN^2 = \NN^2$. Put $\Gamma(\bar{S}) = (G(\bar{S}) \setminus \bar{S}) \cap \NN^2.$ Since $f \in \Gamma(\bar{S}) \subsetneq \cone(\bar{S}),$ it is enough to show that for any term order $\prec,$ $f \neq \max_{\prec} \Gamma(\bar{S})$. Let $S_1$ and $S_2$ denote the projections of $\bar{S}$ onto its first and second components respectively. Then $S_1$ and $S_2$ are numerical semigroups. Let $F(S_1)$ and $F(S_2)$ denote the Frobenius numbers of $S_1$ and $S_2$ respectively. Choose $b_1 \in \NN$ such that $F(S_1)+b_1 = \alpha_1 n_4$, for some $\alpha_1 \in \NN.$ Similarly, choose $a_2 \in \NN$ such that $F(S_2)+a_2 = \alpha_2 n_4$, for some $\alpha_2 \in \NN.$ Then $(F(S_1),b_1), (a_2,F(S_2)) \in \Gamma(\bar{S}).$ Suppose $f = (f_1,f_2) = \max_{\prec} \Gamma(\bar{S})$ with respect to some term order $\prec.$ This implies that $(F(S_1),b_1) \prec f$ and $(a_2,F(S_2)) \prec f$ and hence $(F(S_1)+a_2, F(S_2)+b_1) \prec (2f_1,2f_2).$ Choose $a_2$, $b_1$ such that $F(S_1)+a_2 > 2f_1$ and $F(S_2)+b_1 > 2f_2.$ Since $\prec$ is a term order, it follows that $(2f_1,2f_2) \prec (F(S_1)+a_2, F(S_2)+b_1),$ giving a contradiction to the anti-symmetry of the order.
}\end{Example}

Observe that the above argument holds for any $\MPD$-semigroup in $\NN^2$ with $\beta$-type 1, that is corresponding to the projective closure of a numerical semigroup. 

As is in the case of numerical semigroups, gluing of two almost (pseudo)-symmetric semigroups is not almost (pseudo)-symmetric.

\begin{Remark}
	Gluing of pseudo-symmetric semigroups is not pseudo-symmetric. Let $S = S_1+_d S_2,$ where $S_1$ and $S_2$ are pseudo-symmetric semigroups. Then $\beta$-$\t(S_1) = \beta$-$\t(S_2) = 2$ and therefore $\beta$-$\t(S) = \beta$-$\t(S_1) \cdot \beta$-$\t(S_2) =4$ (see \cite[Theorem 6.1]{hema2019}). Hence, $S$ is not pseudo-symmetric.
\end{Remark}

\begin{Proposition}
	Gluing of two $\prec$-almost symmetric semigroups is not $\prec$-almost symmetric.
\end{Proposition}		

\begin{proof}
	Let $S_1$ and $S_2$ be $\prec$-almost symmetric semigroups with a fixed order $\prec$ and $\beta$-$\t(S_1),$ $\beta$-$\t(S_2) \geq 2.$ Let $S = S_1+_d S_2$ be the gluing of $S_1$ and $S_2$ by $d \in S_1 \cap S_2.$ Suppose that $S$ is $\prec$-almost symmetric.  We claim that $\F(S)_{\substack{\\\prec}} = \F(S_1)_{\substack{\\\prec}} + \F(S_2)_{\substack{\\\prec}} + d.$ Since $S$ is $\prec$-almost symmetric, $\F(S)_{\substack{\\\prec}}$ exists and $\F(S)_{\substack{\\\prec}} \in \PF(S).$ Using Theorem \ref{PFGluing}, we can write $\F(S)_{\substack{\\\prec}} = f_1 + f_2 + d$ for some $f_1 \in \PF(S_1)$ and $f_2 \in \PF(S_2).$ We have to show that $f_1 = \F(S_1)_{\substack{\\\prec}}$ and $f_2 = \F(S_2)_{\substack{\\\prec}}.$ Suppose at least one of $f_i$'s is not the Frobenius element. Without loss of generality, let $f_1 \in \PF(S_1) \setminus \{\F(S_1)_{\substack{\\\prec}}\}$ and $f_2 = \F(S_2)_{\substack{\\\prec}}.$ As $S_1$ is $\prec$-almost symmetric, there exists $g \in \PF(S_1)$ such that $f_1 + g = \F(S_1)_{\substack{\\\prec}}.$ Since $S$ is assumed to be $\prec$-almost symmetric, for every $h_1 + h_2 + d \in \PF(S)$, we have
	\[ F(S)_{\substack{\\\prec}} - (h_1 + h_2 + d) = \F(S_1)_{\substack{\\\prec}} - g + f_2 - h_1 - h_2 \in \PF(S). \]
	This implies that $\F(S_1)_{\substack{\\\prec}} - g + f_2 - h_1 - h_2 + d \in S$ and hence $\F(S_1)_{\substack{\\\prec}} - g + f_2 + d \in S,$ giving a contradiction as $\F(S_1)_{\substack{\\\prec}} - g + f_2 + d \in \PF(S).$ Therefore, the claim holds.
	
	Now, as $\F(S)_{\substack{\\\prec}} - (f+g+d) \in \PF(S),$ for each $f \in \PF(S_1)$ and for each $ g \in \PF(S_2)$, we get 
	\[ (\F(S_1)_{\substack{\\\prec}} -f)+ (\F(S_2)_{\substack{\\\prec}} -g) + d \in S.\] 
	This is a contradiction because $(\F(S_1)_{\substack{\\\prec}} -f) \in \PF(S_1) ,(\F(S_2)_{\substack{\\\prec}} -g) \in \PF(S_2)$ and by Theorem {\rm \ref{PFGluing}} $(\F(S_1)_{\substack{\\\prec}} -f)+ (\F(S_2)_{\substack{\\\prec}} -g) + d \in \PF(S).$ 
\end{proof}

For numerical semigroups, the concept of symmetric and pseudo-symmetric numerical semigroups is also characterized using elements in the gap set, $\ZZ \setminus S$ (see \cite[Proposition 4.4]{numerical}). In order to attempt such characterization in the case of $\MPD$-semigroups, we assume that $\Hc(S)$ is a finite set. Let $S$ be a $\MPD$-semigroup. If $\Hc(S)$ is a non-empty finite set, then $S$ is said to be a $\Cc$-semigroup, where $\Cc$ denotes the cone of the semigroup. Note that the semigroups defined in Example \ref{symmetricexample} are $\Cc$-semigroups. There are $\MPD$-semigroups which are not $\Cc$-semigroups (for example, see Example \ref{RFBinomialexample} and Example \ref{RFrelationexample}). The concept of $\Cc$-semigroups is introduced in \cite{Wilfconjecture} and several properties of these semigroups are investigated in \cite{cSemigroup}. When $S$ is a $\Cc$-semigroup, we give a characterization of $\prec$-symmetric and $\prec$-pseudo-symmetric semigroups.

\begin{Theorem}\label{precsymmetricity}
	Let $S$ be a $\Cc$-semigroup and let $\F(S)_{\substack{\\\prec}}$ denote the Frobenius element of $S$ with respect to an order $\prec.$ 
	Then $S$ is a $\prec$-symmetric semigroup if and only if for each $g \in \cone(S) \cap \NN^d$ we have: \[ g \in S \iff \F(S)_{\substack{\\\prec}} - g \notin S. \]
\end{Theorem}

\begin{proof}
	Suppose $S$ is $\prec$-symmetric semigroup. Assume $g \in S$ and also $\F(S)_{\substack{\\\prec}} - g \in S$. This implies that $\F(S)_{\substack{\\\prec}} \in S$, which is a contradiction. Hence, $g \in S$ implies $\F(S)_{\substack{\\\prec}} - g \notin S.$ Now assume $g \in \cone(S) \cap \NN^d$ and $g \notin S.$ This implies that $g \in \Hc(S).$ By Corollary \ref{MaxPseudoFrob}, there exists $f \in \PF(S)$ such that $f -g \in S.$ Since $S$ is $\prec$-symmetric, it follows that $f = \F(S)_{\substack{\\\prec}}$ and hence $\F(S)_{\substack{\\\prec}} -g \in S.$\\
	Conversely, assume that for each $g \in \cone(S) \cap \NN^d$, $g \in S \iff \F(S)_{\substack{\\\prec}} - g \notin S.$ If $S$ is not $\prec$-symmetric, then there exists $f \in \PF(S)$ not equal to $\F(S)_{\substack{\\\prec}}.$ Since $f \notin S$ implies $\F(S)_{\substack{\\\prec}} - f \in S$ and as $f$ is pseudo-Frobenius element, we get a contradiction. Hence, $S$ is $\prec$-symmetric.
\end{proof}

\begin{Theorem}\label{precPseudosymmetricity}
	Let $S$ be a $\Cc$-semigroup and let $\F (S)_{\substack{\\\prec}}$ denote the Frobenius element of $S$ with respect to an order $\prec.$ Then $S$ is a $\prec$-pseudo-symmetric semigroup if and only if $\F(S)_{\substack{\\\prec}}$ is even, and for each $g \in \cone(S) \cap \NN^d$ we have: 
	\[ g \in S \iff \F(S)_{\substack{\\\prec}} - g \notin S \text{ and } g \neq \F(S)_{\substack{\\\prec}}/2. \]
\end{Theorem}

\begin{proof}
	Let $S$ be $\prec$-pseudo-symmetric. Then $\F(S)_{\substack{\\\prec}}$ is even and $\PF(S)=\{\F(S)_{\substack{\\\prec}}, \F(S)_{\substack{\\\prec}}/2 \}.$ Let $g \in S.$ Then $g \neq \F(S)_{\substack{\\\prec}}/2 $. If $\F(S)_{\substack{\\\prec}} - g \in S$ then $\F(S)_{\substack{\\\prec}} \in S,$ which is a contradiction. Hence $\F(S)_{\substack{\\\prec}} - g \notin S.$ For the converse, we will prove that if $g \in \cone(S) \cap \NN^d$ and $g \notin S$ then $\F(S)_{\substack{\\\prec}} - g \in S$ or $g = \F(S)_{\substack{\\\prec}}/2.$ If $g = \F(S)_{\substack{\\\prec}}/2$ then nothing to prove. If $g \neq \F(S)_{\substack{\\\prec}}/2$, then $g \in \Hc(S)$ and there exists $f \in \PF(S)$ such that $f- g \in S.$ If $f= \F(S)_{\substack{\\\prec}},$ we are done. If not, then $f$ has to be $\F(S)_{\substack{\\\prec}}/2$ and hence $\F(S)_{\substack{\\\prec}}/2 - g \in S $ implying that $\F(S)_{\substack{\\\prec}}/2 + \F(S)_{\substack{\\\prec}}/2 - g = \F(S)_{\substack{\\\prec}} - g \in S.$
	
	Conversely, let $\F(S)_{\substack{\\\prec}}$ be even and for each $g \in \cone(S) \cap \NN^d,$ $g \in S \iff \F(S)_{\substack{\\\prec}} - g \notin S \text{ and } g \neq \F(S)_{\substack{\\\prec}}/2.$ We first prove that $\PF_{\substack{\\\prec}}'(S) = \PF(S) \setminus \{ \F(S)_{\substack{\\\prec}} \} \neq \emptyset.$ In particular, we prove that $\F(S)_{\substack{\\\prec}}/2 \in \PF_{\substack{\\\prec}}'(S).$ Suppose not. Then there exists $s \in S \setminus \{0\}$ such that $\F(S)_{\substack{\\\prec}}/2 + s \notin S.$ This implies that either $\F(S)_{\substack{\\\prec}} - (\F(S)_{\substack{\\\prec}}/2 + s) \in S$ or $\F(S)_{\substack{\\\prec}}/2 + s = \F(S)_{\substack{\\\prec}}/2.$ Since $s \neq 0$, the only possibility is that $\F(S)_{\substack{\\\prec}}/2 - s \in S$ which further implies that $\F(S)_{\substack{\\\prec}}/2 \in S,$ a contradiction. Therefore, $\F(S)_{\substack{\\\prec}}/2 \in \PF_{\substack{\\\prec}}'(S).$ Now, let $f \in \PF_{\substack{\\\prec}}'(S).$ Then as $f \notin S,$ it follows that either $\F(S)_{\substack{\\\prec}} - f \in S$ or $f = \F(S)_{\substack{\\\prec}}/2.$ If $\F(S)_{\substack{\\\prec}} - f \in S,$ then it implies that $\F(S)_{\substack{\\\prec}} \in S$, giving a contradiction. Therefore, $f = \F(S)_{\substack{\\\prec}}/2.$ Hence, $S$ is $\prec$-pseudo-symmetric.
\end{proof}

\begin{Theorem}
	Let $S_1$ and $S_2$ be $\prec$-symmetric semigroups with a fixed term order $\prec.$ Let $S$ be a gluing of $S_1$ and $S_2$ such that $S$ is a $\Cc$-semigroup. Then $S$ is a $\prec$-symmetric semigroup.
\end{Theorem}
\begin{proof}
	Let $S = S_1 +_d S_2$ and let $\F(S_1)_{\substack{\\\prec}}$ and $\F(S_2)_{\substack{\\\prec}}$ be the Frobenius elements of $S_1$ and $S_2$ respectively. Using Theorem \ref{PFGluing}, we get $\F(S_1)_{\substack{\\\prec}} + \F(S_2)_{\substack{\\\prec}} + d \in \PF(S).$ Since $\beta$-$\t(S)= \beta$-$\t(S_1) \cdot \beta$-$\t(S_2)$ (see \cite[Theorem 6.1]{hema2019}), we have $|\PF(S)| = 1.$ Therefore, it is enough to show that $\F(S_1)_{\substack{\\\prec}} + \F(S_2)_{\substack{\\\prec}} + d = \text{max}_{\prec} \Hc(S).$ Suppose that there exists $g \in \Hc(S)$ such that $\F(S_1)_{\substack{\\\prec}} + \F(S_2)_{\substack{\\\prec}} + d \prec g.$ Since $\F(S_1)_{\substack{\\\prec}} + \F(S_2)_{\substack{\\\prec}} + d \in \PF(S)$ and as $|\PF(S)|=1$, using Corollary \ref{MaxPseudoFrob} we get $g \leq \F(S_1)_{\substack{\\\prec}} + \F(S_2)_{\substack{\\\prec}} + d.$ Therefore, $\F(S_1)_{\substack{\\\prec}} + \F(S_2)_{\substack{\\\prec}} + d - g \in S.$ Write $\F(S_1)_{\substack{\\\prec}} + \F(S_2)_{\substack{\\\prec}} + d - g = s$ for some $s \in S$. Since $\F(S_1)_{\substack{\\\prec}} + \F(S_2)_{\substack{\\\prec}} + d \prec g$, we get $g+s \prec g$. If $s \neq 0$, then $0 \prec s$ implies that $g \prec s+g$. Therefore, $\F(S_1)_{\substack{\\\prec}} + \F(S_2)_{\substack{\\\prec}} + d = g$ and hence $\F(S_1)_{\substack{\\\prec}} + \F(S_2)_{\substack{\\\prec}} + d = \text{max}_{\prec} \Hc(S).$
\end{proof}

The following result gives a different characterization of $\prec$-symmetric and $\prec$-pseudo-symmetric $\Cc$-semigroups when $\cone(S) \cap \NN^d = \NN^d.$ On $\cone(S),$ we define a relation $\leq_c$ as follows: $g \leq_c f$ if $g_i \leq f_i$ for all $i \in [1,d].$

\begin{Theorem}\label{precsymmetricCardH(S)}
	Let $S$ be a $\Cc$-semigroup such that $\cone(S)  \cap \NN^d = \mathbb{N}^d.$ Then
	\begin{enumerate}[{\rm (1)}]
		\item $S$ is $\prec$-symmetric if and only if $| \Hc(S)| = | \{ g \in S \mid g \leq_c \F(S)_{\substack{\\\prec}}\}|.$
		
		\item $S$ is $\prec$-pseudo-symmetric if and only if $| \Hc(S) \setminus \{ \F(S)_{\substack{\\\prec}}/2 \}| = | \{ g \in S \mid g \leq_c \F(S)_{\substack{\\\prec}}\}|$ and $\F(S)_{\substack{\\\prec}}$ is even. 
	\end{enumerate}  
\end{Theorem}

\begin{proof}
    (1) Let $S$ be $\prec$-symmetric. We define a map
    \begin{align*}
    \phi : \Hc(S) \longrightarrow \{g \in S \mid g \leq_c \F(S)_{\substack{\\\prec}} \} \text{ by } \phi(g) =  \F(S)_{\substack{\\\prec}} - g.
    \end{align*}

	Let $g \in \Hc(S).$ Since $S$ is $\prec$-symmetric and $g \notin S,$ by Theorem \ref{precsymmetricity} we get $\F(S)_{\substack{\\\prec}} - g \in S.$ Also as $g \in \NN^d$, we have $\F(S)_{\substack{\\\prec}} - g \leq_c \F(S)_{\substack{\\\prec}}.$ Thus $\phi$ is a well defined map and is clearly injective. Now let $g \in S$ such that $g \leq_c \F(S)_{\substack{\\\prec}}.$ Then from Theorem \ref{precsymmetricity}, we get $\F(S)_{\substack{\\\prec}} - g \notin S$ and  $\F(S)_{\substack{\\\prec}} - g \in \NN^d = \cone(S) \cap \NN^d.$ Thus $\F(S)_{\substack{\\\prec}} - g \in \Hc(S).$ Therefore, $\phi$ is a bijective map.
	
	Conversely, let $| \Hc(S)| = | \{ g \in S \mid g \leq_c \F(S)_{\substack{\\\prec}}\}|.$ Consider the map 
	\begin{align*}
		\psi : \{ g \in S \mid g \leq_c \F(S)_{\substack{\\\prec}}\} \longrightarrow \Hc(S) \text{ such that } \psi(g) = \F(S)_{\substack{\\\prec}} - g.
	\end{align*}

	Then $\psi$ is a well defined map as for any $g \in S$ such that $g \leq_c \F(S)_{\substack{\\\prec}},$ we have $\F(S)_{\substack{\\\prec}} - g \notin S$ and $\F(S)_{\substack{\\\prec}} - g \in \NN^d = \cone(S) \cap \NN^d.$ Thus $\F(S)_{\substack{\\\prec}} - g \in \Hc(S).$ Clearly $\psi$ is an injective map. Since  $| \Hc(S)| = | \{ g \in S \mid g \leq_c \F(S)_{\substack{\\\prec}}\}|,$ we get $\psi$ is a bijective map. Observe that $\psi^{-1} : \Hc(S) \longrightarrow \{ g \in S \mid g \leq_c \F(S)_{\substack{\\\prec}}\},$ defined by $\psi^{-1}(g) = \F(S)_{\substack{\\\prec}} - g$ for all $g \in \Hc(S).$ Now, in order to prove that $S$ is $\prec$-symmetric, we use Theorem \ref{precsymmetricity}. Let $g \in \cone(S) \setminus S.$ Then $g \in \Hc(S)$ implies that $\psi^{-1}(g) = \F(S)_{\substack{\\\prec}} - g \in S.$ Now for $g \in \cone(S),$ let $\F(S)_{\substack{\\\prec}} - g \in S.$ Since $\F(S)_{\substack{\\\prec}} - g \in \NN^d$ and $g \in \NN^d,$ we get $\F(S)_{\substack{\\\prec}} - g \leq_c \F(S)_{\substack{\\\prec}}.$ Therefore $\psi(\F(S)_{\substack{\\\prec}} - g) = g \in \Hc(S).$ This implies $g \notin S$ and hence $S$ is $\prec$-symmetric.

    (2) The proof follows using the same arguments as in part (1) and using Theorem \ref{precPseudosymmetricity}. 
\end{proof}

\begin{Definition}[{\cite[Definition 1]{Wilfconjecture}}]
	Let $S$ be a $\Cc$-semigroup. Define the Frobenius number of $S$ as $\mathcal{N}(\F(S)_{\substack{\\\prec}}) = | \Hc(S)| + | \{ g \in S \mid g \prec \F(S)_{\substack{\\\prec}}\}|.$
\end{Definition}

Observe that if $S$ is a numerical semigroup, then $\mathcal{N}(\F(S)_{\substack{\\\prec}}) = \F(S),$ the Frobenius number of $S.$ For a numerical semigroup $S,$ Wilf proposed a conjecture related to the Diophantine Frobenius Problem that claims that the inequality 
\[ \F(S)+1 \leq e(S) \cdot |\{ s \in S \mid s < \F(S) \}| \] 
is true. While this conjecture still remains open, a potential extension of Wilf's conjecture to affine semigroups is studied in \cite{Wilfconjecture}.

\begin{Conjecture} [{\bf Extension of Wilf's conjecture} {\cite[Conjecture 14]{Wilfconjecture}}]
	Let $S$ be a $\Cc$-semigroup. The extended Wilf's conjecture is
	\[ | \{ g \in S \mid g \prec \F(S)_{\substack{\\\prec}} \} | \cdot e(S) \geq \mathcal{N}(\F(S)_{\substack{\\\prec}}) + 1, \]
	where $e(S)$ denotes the embedding dimension of $S.$
\end{Conjecture}

Note that if $S$ is an affine semigroup of $\NN^d,$ where $d \geq 2,$ then the semigroup ring $k[S]$ is Cohen-Macaulay when $e(S)=2.$ Since our affine semigroups are $\MPD$, we may assume that $e(S) \geq 3.$

\begin{Theorem}  \label{Wilf}
	Let $S$ be a $\Cc$-semigroup such that $\cone(S)  \cap \NN^d = \mathbb{N}^d,$ $d >1.$ If $S$ is a $\prec$-symmetric or a $\prec$-pseudo-symmetric semigroup, then the extended Wilf's conjecture holds.
\end{Theorem} 

\begin{proof}
Observe that
\begin{center}
$ \{ g \in S \mid g \leq_c \F(S)_{\substack{\\\prec}} \}  \subseteq \{ g \in S \mid g \prec \F(S)_{\substack{\\\prec}} \},$ for any term order $\prec.$
\end{center} 
	If $S$ is $\prec$-symmetric, then from Theorem \ref{precsymmetricCardH(S)} we see that
	\begin{align*} 
	\mathcal{N}(\F(S)_{\substack{\\\prec}}) 
	= | \Hc(S)| + | \{ g \in S \mid g \prec \F(S)_{\substack{\\\prec}}\}| 
	& \leq 2~| \{ g \in S \mid g \prec \F(S)_{\substack{\\\prec}}\}| \\
	&< e(S) \cdot | \{ g \in S \mid g \prec \F(S)_{\substack{\\\prec}}\}|.
	\end{align*}
	Hence, $\mathcal{N}(\F(S)_{\substack{\\\prec}}) + 1 \leq e(S) \cdot | \{ g \in S \mid g \prec \F(S)_{\substack{\\\prec}}\}|.$ 
	
	Now, let $S$ be $\prec$-pseudo-symmetric. Using Theorem \ref{precsymmetricCardH(S)} we get 
	\begin{align*}
		\mathcal{N}(\F(S)_{\substack{\\\prec}}) + 1
		= | \Hc(S)| + | \{ g \in S \mid g \prec \F(S)_{\substack{\\\prec}}\}| + 1
		\leq 2~| \{ g \in S \mid g \prec \F(S)_{\substack{\\\prec}} \} | + 2.
	\end{align*}
	If $| \{ g \in S \mid g \prec \F(S)_{\substack{\\\prec}} \} | \geq 2,$ then $2~| \{ g \in S \mid g \prec \F(S)_{\substack{\\\prec}} \} | + 2 \leq e(S) \cdot | \{ g \in S \mid g \prec \F(S)_{\substack{\\\prec}}\}|$ and thus $S$ satisfies the extended Wilf's conjecture. The conjecture is also satisfied when $| \{ g \in S \mid g \prec \F(S)_{\substack{\\\prec}}\}| = 1$ and $e(S) \geq 4.$ We now prove that the case $| \{ g \in S \mid g \prec \F(S)_{\substack{\\\prec}}\}| = 1$ and $e(S)=3$ is not possible, which will complete the proof. \\
	Suppose $| \{ g \in S \mid g \prec \F(S)_{\substack{\\\prec}}\}| = 1.$ Since ${\bf 0} = (0,0,\ldots,0) \prec \F(S)_{\substack{\\\prec}},$ using Theorem \ref{precsymmetricCardH(S)} we get $\Hc(S) = \{ \F(S)_{\substack{\\\prec}}/2, \F(S)_{\substack{\\\prec}} \}.$ In particular, $S = \NN^d \setminus \{ \F(S)_{\substack{\\\prec}}/2, \F(S)_{\substack{\\\prec}} \}.$ As $S$ is a pseudo-symmetric semigroup, the only possible choice for the Frobenius number in this setup is $\F(S)_{\substack{\\\prec}} = 2e_i,$ for some $i=1,\ldots,d,$ where $\{ e_1,\ldots,e_d \}$ is a standard $\QQ$-basis of vector space $\QQ^d.$ Thus, $S = \NN^d \setminus \{ e_i, 2e_i \}$ for some $i \in \{1,\ldots,d\}$ and hence is minimally generated by 
	\[ \{ 3e_i, 4e_i, 5e_i\} \bigcup \left( \bigcup_{\substack{j=1\\ j \neq i}}^{d} \{e_j\} \right) \bigcup \left( \bigcup_{\substack{j=1\\ j \neq i}}^{d} \{ e_i+e_j, 2e_i + e_j \} \right). \]
	Since $d \geq 2,$ it follows that the embedding dimension in this case is at least $6.$ Thus, $e(S)$ cannot be 3 whenever $| \{ g \in S \mid g \prec \F(S)_{\substack{\\\prec}}\}| = 1.$
\end{proof}

\section{Hilbert Series and Pseudo-Frobenius Elements}

In the case of numerical semigroups, finding the Frobenius number is a classical problem, known as the Diophantine Frobenius Problem. A number of methods, from several areas of mathematics, have been used to find a formula giving the Frobenius number. In one such attempt in \cite{frobenius}, the authors equate the Frobenius number $\F(S)$ of a numerical semigroup $S$ with the $a$-invariant of the semigroup ring $k[S].$

Recall that the Hilbert series of the $S$-graded ring $k[S]$ is 
\[ \H(k[S]; {\bf t}) = \sum_{s \in S} \dim_k (k[S]_s) ~{\bf t}^a, \]
where $k[S]_s$ is the $k$-vector space generated by all monomials of $S$-degree $s$ and ${\bf t} = (t_1,\ldots,t_d)$ is a set of indeterminates. By Hilbert-Serre theorem, we know that the Hilbert series of $k[S]$ can be expressed as a rational function
\[	\H(k[S]; {\bf t}) = \frac{\mathcal{K}(k[S]; {\bf t})}{\prod_{i=1}^n(1-{\bf t}^{a_i})}. \]
The numerator $\mathcal{K}(k[S]; {\bf t}) \in \ZZ[{\bf t}]$ is known as the $K$-polynomial of $k[S]$ and when $S$ is a numerical semigroup, the degree of $\H(k[S]; {\bf t}),$ as a rational function, is called the $a$-invariant of $k[S]$ and it is denoted by $a(k[S]).$ In this section, we prove an analogue of this connection for a $\Cc$-semigroup $S$ with $\cone(S) \cap \NN^d = \NN^d$ and derive the result \cite[Theorem 3.1]{frobenius} as a consequence.

\begin{Theorem}
	Let $S=\langle a_1, \ldots, a_n \rangle \subset \mathbb{N}^d$ be a $\MPD$-semigroup. Then 
	\begin{enumerate}[{\rm(1)}]
	\item $\displaystyle \mathcal{K}(k[S];{\bf t}) 
		= \sum_{a \in S} \sum_{i=0}^{n-2}(-1)^i \beta_{i,a}(k[S]){\bf t}^a + \sum_{a\in S} (-1)^{n-1}\beta_{n-1,a}(k[S]){\bf t}^a,$ \\[1mm]
where $\beta_{i,a}$ is the $i^{th}$-multigraded Betti number of $k[S]$ in degree $a$, and

	\item $\PF(S) = \{ a - \sum_{i=1}^n a_i \mid \beta_{n-1,a} \neq 0\}.$
	\end{enumerate}
	
	In particular, the pseudo-Frobenius elements of $S$ can be obtained from the Hilbert series of the semigroup ring $k[S]$.
\end{Theorem}

\begin{proof}
	Let $\H(k[S];{\bf t}) = \sum_{a \in S} \dim_k(k[S]_a){\bf t}^a$ denote the Hilbert series of $k[S],$ where ${\bf t} = (t_1,\ldots,t_d).$ Then there exists a unique polynomial $\mathcal{K}(k[S];{\bf t}) \in \mathbb{Z}[{\bf t}]$ (see \cite{millersturmfels}) such that
	\[
	\H(k[S];{\bf t}) = \frac{\mathcal{K}(k[S];{\bf t})}{\prod_{i=1}^n (1-{\bf t}^{a_i})}.
	\]
	Using \cite[Proposition 8.23]{millersturmfels}, $\mathcal{K}(k[S];{\bf t})$ can be written as: 
	\[
	\mathcal{K}(k[S];{\bf t}) = \sum_{a \in S,i \geq 0} (-1)^i \beta_{i,a}(k[S]){\bf t}^a.
	\]
	 Since $S$ is a $\MPD$-semigroup, we have $\pdim_R k[S]=n-1$ and therefore $\beta_{i,a}(k[S]) = 0$ for $i > n-1$. Hence,
	\begin{align*}
		\mathcal{K}(k[S];{\bf t}) &=  \sum_{a \in S} \sum_{i=0}^{n-1}(-1)^i \beta_{i,a}(k[S]){\bf t}^a  \\
		&= \sum_{a \in S} \sum_{i=0}^{n-2}(-1)^i \beta_{i,a}(k[S]){\bf t}^a + \sum_{a\in S} (-1)^{n-1}\beta_{n-1,a}(k[S]){\bf t}^a.
	\end{align*}
	
	Now from \cite[Corollary 7]{pfelements}, we have $\beta_{n-1,a} \neq 0$ if and only if $a \in \{f + \sum_{i=1}^n a_i \mid f \in \PF(S)\}$ and hence, $\PF(S) = \{ a - \sum_{i=1}^n a_i \mid \beta_{n-1,a} \neq 0\}$ can be obtained from $\H(k[S];{\bf t})$.
\end{proof}

For a monomial ${\bf t}^a \in k[t_1, \ldots,t_n]$, $\exp({\bf t}^a)$ denotes the exponent $a.$ Let $ f = \sum_{a} \gamma_a {\bf t}^a$ be a non-zero polynomial in $k[t_1,\ldots,t_n]$ and let $\prec$ be a term order. The leading term of $f$ with respect to $\prec$ is defined as $LT_{\prec}(f) := \gamma_{\text{multideg}(f)} {\bf t}^{\text{multideg}(f)}$, where $\text{multideg}(f) = \max_{\prec} \{ a \mid \gamma_a \neq 0 \}.$

\begin{Theorem}  \label{FrobExp}
	Let $S=\langle a_1, \ldots, a_n \rangle \subset \mathbb{N}^d$ be a $\Cc$-semigroup such that $\cone(S)  \cap \NN^d =\mathbb{N}^d.$ Then $\F(S)_{\substack{\\\prec}} = \exp (LT_{\prec} \ \mathcal{K}(k[S];{\bf t})) - \sum_{i=1}^n a_i$ for any term order $\prec.$
\end{Theorem}

\begin{proof}
	We have
	\[
	\frac{1}{\prod_{i=1}^d(1-t_i)} = \sum_{a \in \mathbb{N}^d} {\bf t}^a. 
	\]
	Since $\cone(S)  \cap \NN^d = \mathbb{N}^d$, we can write
	\begin{align*}
		\frac{1}{\prod_{i=1}^d(1-t_i)} = \sum_{a\in S}{\bf t}^a + \sum_{a \in \Hc(S)} {\bf t}^a &= \H(k[S];{\bf t})+\sum_{a \in \Hc(S)}{\bf t}^a  \\
		&= \frac{\mathcal{K}(k[S],{\bf t})}{\prod_{i=1}^n(1-{\bf t}^{a_i})}+\sum_{a \in \Hc(S)}{\bf t}^a.
	\end{align*}
	Therefore, we have 
	\[
	\mathcal{K}(k[S];{\bf t})\prod_{i=1}^d(1-t_i) =
	\prod_{i=1}^n(1-{\bf t}^{a_i}) - \prod_{i=1}^d(1-t_i)\prod_{i=1}^n(1-{\bf t}^{a_i}) \bigg( \sum_{a \in \Hc(S)}{\bf t}^a \bigg).
	\] 
	Since $\Hc(S)$ is finite, it follows that for each term order $\prec$ on $\mathbb{N}^d$, $\max_{\prec} \Hc(S)$ exists and is a Frobenius element of $S$. Therefore,
	\[
	LT_{\prec} \Bigg(\mathcal{K}(k[S];{\bf t})\prod_{i=1}^d(1-t_i) \Bigg) = LT_{\prec} \Bigg(\prod_{i=1}^d(1-t_i)\prod_{i=1}^n(1-{\bf t}^{a_i}) \Big(\sum_{a \in \Hc(S)}{\bf t}^a \Big) \Bigg).
	\]
	Hence,
	\[
	\exp (LT_{\prec}(\mathcal{K}(k[S];{\bf t})) = \F(S)_{\substack{\\\prec}}+\sum_{i=1}^n a_i.
	\]
\end{proof}

In \cite{frobenius}, the authors proved that the Frobenius number of a numerical semigroup $S$ is the $a$-invariant of the ring $k[S].$ Since any numerical semigroup $S$ is a $\MPD$-semigroup, with $\cone(S) \cap \NN = \NN$ and $\Hc(S) = \NN \setminus S,$ the result \cite[Theorem 3.1]{frobenius} now follows as a corollary of Theorem \ref{FrobExp}.

\begin{Example}
	Let $S = \langle a_1=(0,1),a_2=(2,0),a_3=(3,0),a_4=(1,3) \rangle.$ Then $\cone(S)  \cap \NN^2 = \mathbb{N}^2$ and $\Hc(S)= \{(1,0),(1,1),(1,2)\}$ is finite. Therefore, $S$ a symmetric $\MPD$-semigroup with $\F(S)_{\substack{\\\prec}} = (1,2)$ for any term order $\prec.$ We have,
	\[
	\H(k[S];{\bf t}) = \frac{1-t_1^6-t_1^3t_2^3-t_1^4t_2^3-t_1^2t_2^6+t_1^6t_2^3+t_1^7t_2^3+t_1^4t_2^6+t_1^5t_2^6-t_1^7t_2^6}{(1-t_1t_2^3)(1-t_1^3)(1-t_1^2)(1-t_2)}.
	\]
	Therefore,
	\[ 
	\mathcal{K}(k[S];{\bf t}) = 1-t_1^6-t_1^3t_2^3-t_1^4t_2^3-t_1^2t_2^6+t_1^6t_2^3+t_1^7t_2^3+t_1^4t_2^6+t_1^5t_2^6-t_1^7t_2^6
	\]
	and hence, $\F(S)_{\substack{\\\prec}} = \exp (LT_{\prec} \ \mathcal{K}(k[S];{\bf t})) - \sum_{i=1}^4 a_4 = (7,6) - (6,4) = (1,2)$.
\end{Example}

\begin{Corollary}
	Let $S_1$ and $S_2$ be two $\Cc$-semigroups of $\mathbb{N}^d$ such that $\cone(S_1)  \cap \NN^d =\mathbb{N}^d = \cone(S_2)  \cap \NN^d.$ Then $S_1$ and $S_2$ can not be glued.
\end{Corollary}

\begin{proof}
	Let $S_1 = \langle a_1,\ldots,a_e \rangle,$ $S_2 = \langle a_{e+1},\ldots,a_n \rangle$ and suppose $S = S_1 +_s S_2$ be the gluing of $S_1$ and $S_2$ by $s$. Using \cite[Theorem 4.3]{AssiA}, we get
	\begin{align*}
		\H(k[S];{\bf t}) 
		&= (1 - {\bf t}^s) ~\H(k[S_1]; {\bf t}) ~\H(k[S_2]; {\bf t})  \\
		&= (1 - {\bf t}^s) ~\frac{\mathcal{K}(k[S_1]; {\bf t})}{\prod_{i=1}^{e}(1 - {\bf t}^{a_i})} ~\frac{\mathcal{K}(k[S_2]; {\bf t})}{\prod_{i=e+1}^{n}(1 - {\bf t}^{a_i})}.
	\end{align*}
	 For any term order $\prec$, we have
	\begin{align*}
		\exp(LT_{\prec} \ \mathcal{K}(k[S]; {\bf t})) 
		&= \exp\bigg(LT_{\prec} \Big( (1-{\bf t}^s) ~\mathcal{K}(k[S_1]; {\bf t}) ~\mathcal{K}(k[S_2; {\bf t}]) \Big) \bigg) \\
		&= s + \exp(LT_{\prec} \ \mathcal{K}(k[S_1]; {\bf t})) + \exp(LT_{\prec} \ \mathcal{K}(k[S_2]; {\bf t})).
	\end{align*}
	Since $\cone(S_1)  \cap \NN^d =\mathbb{N}^d = \cone(S_2)  \cap \NN^d$, we have $\cone(S) \cap \NN^d = \NN^d.$ Therefore, $S$ is a $\Cc$-semigroup. From Theorem \ref{FrobExp} it follows that
	\begin{align*}
		\F(S)_{\substack{\\\prec}} 
		&= \exp(LT_{\prec} \ \mathcal{K}(k[S]; {\bf t})) - \sum_{i=1}^{n} a_i  \\
		&= s + \exp(LT_{\prec} \ \mathcal{K}(k[S_1]; {\bf t})) + \exp(LT_{\prec} \ \mathcal{K}(k[S_2]; {\bf t})) - \sum_{i=1}^{n} a_i  \\
		& = s + \bigg( \exp(LT_{\prec} \ \mathcal{K}(k[S_1]; {\bf t})) - \sum_{i=1}^{e} a_i \bigg) + \bigg( \exp(LT_{\prec} \ \mathcal{K}(k[S_2]; {\bf t})) - \sum_{i=e+1}^{n} a_i \bigg)  \\
		&= s + \F(S_1)_{\substack{\\\prec}} + \F(S_2)_{\substack{\\\prec}}.
	\end{align*}
 Now, observe that $\Hc(S) \subseteq \Hc(S_1)$. This implies  
 \[ \F(S)_{\substack{\\\prec}}= s + \F(S_1)_{\substack{\\\prec}} + \F(S_2)_{\substack{\\\prec}} = \max_{\prec} \Hc(S) \preceq \max_{\prec} \Hc(S_1) =  \F(S_1)_{\substack{\\\prec}}, \] which is absurd.
\end{proof}

\section{$\RF$-matrices of submonoids of $\NN^d$}

Let $S = \langle a_1, \ldots, a_n \rangle$ be a $\MPD$-semigroup in $\NN^d,$ minimally generated by  $ a_1, \ldots, a_n.$ We recall the notion of row-factorization matrix ($\RF$-matrix), introduced by A. Moscariello in \cite{moscariello}.

\begin{Definition}
	Let $f \in \PF(S).$ An $n \times n$ matrix $M = (m_{ij})$ is an $\RF$-matrix of $f$ if
	$m_{ii} = -1$ for every $i$, $m_{ij} \in \NN$ if $i \neq j$ and for every $i = 1, \ldots ,n$, $\sum_{j=1}^n m_{ij} a_j = f.$
\end{Definition}

While $\RF$-matrices were defined for pseudo-Frobenius elements in numerical semigroups, using Definition \ref{pf}, we observe that the above definition holds in the case of pseudo-Frobenius elements in $\MPD$-semigroups over $\NN^d.$ Note that an $\RF$-matrix of $f$ need not be uniquely determined. Thus, the notation $\RF(f)$ will denote one of the possible $\RF$-matrices of $f.$

In \cite{etoRowFactor}, Eto proved that when $S$ is a numerical semigroup, the determinant of an $\RF$-matrix of a pseudo-Frobenius element $f$ is a multiple of $f.$ We study the behavior of the determinant of a row-factorization matrix of a pseudo-Frobenius element of a $\MPD$-semigroup in the next result. 

\begin{Theorem} \label{detRF}
	If $\rank G(S) > 1,$ then the determinant of row-factorization matrices is zero. 
\end{Theorem}

\begin{proof}
	Let $S = \langle a_1,\ldots,a_n \rangle ~ \subseteq \mathbb{N}^d$ be a $\MPD$-semigroup, We can define a natural map
	\[
	\phi: \mathbb{Z}^n \rightarrow \mathbb{Z}^d ~\text{such that} ~ (\alpha_1,\alpha_2,\ldots,\alpha_n) \mapsto \sum_{j=1}^n \alpha_ja_j.
	\]
	Let $L$ be the kernel of the map $\phi$. Observe that $\Im (\phi) = G(S).$ Thus, we have an exact sequence
	\[
	0 \rightarrow L \rightarrow \mathbb{Z}^n \xrightarrow{\phi} G(S) \rightarrow 0. 
	\]
	Since $G(S)$ is free and finitely generated as $\mathbb{Z}$-module, we get $\rank(L) = n-\rank G(S)$. Let $f \in \PF(S)$ and $\RF(f)$ be an $\RF$-matrix of $f$. Let $R_1,R_2,\ldots,R_n$ denote the rows of $\RF(f)$. Consider matrices
	\begin{align*}
		M =
		\left[
		\begin{array}{c}
			R_2-R_1\\
			R_3-R_1\\
			\vdots \\
			R_n-R_1
		\end{array}
		\right] 
		\text{  and  }
		M' = 
		\begin{bmatrix}
			R_1 \\
			M 
		\end{bmatrix}.
	\end{align*}
	
	Note that each row of $M$ is an element of $L$. Hence, $\rank (M) \leq n- \rank G(S).$
	Since $\rank (M') \leq \rank (M) + 1$ and $\mathrm{rowspace} (M') = \mathrm{rowspace} (\RF(f))$, it follows that 
	\[ \rank \RF(f) \leq \rank (M) + 1 \leq n- \rank G(S)+1. \] 
	Hence, $\RF(f)$ can have non-zero determinant only if $\rank G(S)= 1$.
\end{proof}

\begin{Theorem}
	Let $f, f' \in \PF(S)$ with $f + f' \notin S.$ Set $\RF(f) = M = (m_{ij})$ and $\RF(f') = M' = (m'_{ij}).$ Then either $m_{ij} = 0$ or $m'_{ji} = 0$ for every pair $i \neq j.$ 
\end{Theorem}

\begin{proof}
	Same as proof of \cite[Lemma 3.2]{herzogWatanabe}.
\end{proof}

\begin{Example}
Let $S = \langle (0,1),(3,0),(4,0),(1,4),(5,0),(2,7) \rangle$. Observe that 
\[
\Hc(S) = \{(1,0),(2,0),(1,1),(2,1),(1,2),(2,2),(1,3),(2,3),(2,4),(2,5),(2,6)\}.
\]
Therefore, $S$ is a $\Cc$-semigroup. Note that $\PF(S)= \{(1,3),(2,6)\}.$ Take $f = f' = (1,3).$ Then we have
\begin{align*}
	\RF(1,3) =
	\left[
	\begin{array}{c c c c c c}
		-1 & 0 & 0 & 1 & 0 & 0\\
		3 & -1 & 1 & 0 & 0 & 0\\
		3 & 0 & -1 & 0 & 1 & 0\\
		0 & 0 & 0 & -1 & 0 & 1\\
		3 & 2 & 0 & 0 & -1 & 0\\
		10 & 1 & 0 & 0 & 0 & -1
	\end{array}
	\right].
\end{align*}
\end{Example}

For a vector $b = (b_1,\ldots,b_n) \in \ZZ^n$, we let $b^{+}$ denote the vector whose $i$-th entry is $b_i$ if $b_i \geq 0$, and is zero otherwise, and we let $b^{-} = b^{+} - b.$ Then $b = b^{+} - b^{-}$ with $b^{+}, b^{-} \in \NN^n.$ As expected, for a pseudo-Frobenius element $f$, the rows of $\RF(f)$ produce binomials in $I_S.$ Let $e_1,\ldots,e_n$ be the canonical unit vectors of $\ZZ^n.$

\begin{Theorem}  \label{RFBinomial}
	Let $m_1, \ldots, m_n$ be the row vectors of $\RF(f)$, and set $m_{(ij)} = m_i - m_j$ for all $i, j$ with $1 \leq i < j \leq n.$ Then $\phi_{ij} = {\bf x}^{m_{(ij)}^+} - {\bf x}^{m_{(ij)}^-} \in I_S$ for all $i < j.$ Moreover, $\deg \phi_{ij} \leq f + a_i + a_j$ (component-wise). Equality holds, if the vectors $m_i + e_i + e_j$ and $m_j + e_i + e_j$ have disjoint support, in which case there is no cancellation when taking the difference
	of these two vectors.
\end{Theorem}

\begin{proof}
	Same as proof of \cite[Lemma 3.4]{herzogWatanabe}.
\end{proof}

\begin{Example}\label{RFBinomialexample}
Let $S = \langle (1,3),(1,5),(2,1),(2,3),(5,1) \rangle$. Observe that $S$ is not a $\Cc$-semigroup. By using Macaulay2, we have $\PF(S) = \{(5,13),(6,11),(9,6)\}$ and 
\begin{align*}
	\RF(5,13) =
	\left[
	\begin{array}{ c c c c c}
		-1 & 2 & 0 & 2 & 0\\
		6 & -1 & 0 & 0 & 0\\
		0 & 1 & -1 & 3 & 0\\
		5 & 0 & 1 & -1 & 0\\
		0 & 2 & 4 & 0 & -1
	\end{array}
	\right].
\end{align*}

Now consider all the choices of pairs of rows $m_i , m_j$ such that $i<j$. Then the resulting $m_{(ij)}$ and $\phi_{ij}$ are
\begin{table}[H]
	\begin{tabular}{p{2.25in}p{2.25in}p{2.25in}}
		$m_{(12)} = (-7,3,0,2,0)$, & $\phi_{12} = x_2^3x_4^2-x_1^7$, & $\deg(\phi_{12}) = (7,21)$ \\
		$m_{(13)} = (-1,1,1,-1,0)$, & $\phi_{13} = x_2x_3-x_1x_4$, & $\deg(\phi_{13}) = (3,6)$ \\
		$m_{(14)} = (-6,2,-1,3,0)$, & $\phi_{14} = x_2^2x_4^3-x_1^6x_3$, & $\deg(\phi_{14}) = (8,19)$ \\
		$m_{(15)} = (-1,0,-4,2,1)$, & $\phi_{15} = x_4^2x_5-x_1x_3^4$, & $\deg(\phi_{15}) = (9,7)$ \\
		$m_{(23)} = (6,-2,1,-3,0)$, & $\phi_{23} = x_1^6x_3-x_2^2x_4^3$, & $\deg(\phi_{23}) = (8,19)$ \\
		$m_{(24)} = (1,-1,-1,1,0)$, & $\phi_{24} = x_1x_4-x_2x_3$, & $\deg(\phi_{24}) = (3,6)$ \\
		$m_{(25)} = (6,-3,-4,0,1)$, & $\phi_{25} = x_1^6x_5-x_2^3x_3^4$, & $\deg(\phi_{25}) = (11,19)$ \\
		$m_{(34)} = (-5,1,-2,4,0)$, & $\phi_{34} = x_2x_4^4-x_1^5x_3^2$, & $\deg(\phi_{34}) = (9,17)$ \\
		$m_{(35)} = (0,-1,-5,3,1)$, & $\phi_{35} = x_4^3x_5-x_2x_3^5$, & $\deg(\phi_{35}) = (11,10)$ \\
		$m_{(45)} = (5,-2,-3,-1,1)$, & $\phi_{45} = x_1^5x_5-x_2^2x_3^3x_4$, & $\deg(\phi_{45}) = (10,16).$
	\end{tabular}
\end{table}

Observe that for $i=3$ and $j=4,$ the vectors $m_3 + e_3 + e_4$ and $m_4 + e_3 +e_4$ have disjoint support. Hence, $\deg (\phi_{34}) = f+a_3+a_4 = (5,13)+(2,1)+(2,3) = (9,17).$ Cancellation occurs in computing $m_{45}$ and degree of $\phi_{45}$ is $(10,16)$ which is less than $(12,17)=f+a_4+a_5.$

\end{Example}

\begin{Definition}
	The binomials $\phi_{ij}$ defined in Theorem {\rm \ref{RFBinomial}} are called $\RF(f)$-relations. We call a binomial relation $\phi \in I_S$ an \textbf{$\RF$-relation} if it is an $\RF(f)$-relation for some $f \in \PF(S)$.
\end{Definition}

\begin{Example}\label{RFrelationexample}
	Let $S = \langle (2,11),(3,0),(5,9),(7,4) \rangle.$ Note that $\Hc(S)$ is not finite and hence $S$ is not $\Cc$-semigroup. By Macaulay2, we see that 
\begin{center}
$
0 \longrightarrow R(-(94,82)) \oplus R(-(81,93)) \longrightarrow R^6 \longrightarrow R^5 \longrightarrow R \longrightarrow \frac{R}{I_S} \longrightarrow 0
$
\end{center}
 is a minimal graded free resolution of $k[S]$ over $R$, where $R = k[x_1,x_2,x_3,x_4]$. Therefore, $S$ is a $\MPD$-semigroup and the degrees of minimal generators of the  second syzygy module are $(94,82)$ and $(81,93)$. 
Hence $\PF(S) = \{(64,69),(77,58)\}$ and the $\RF$-matrices $\RF(64,69), \RF(77,58)$ are respectively
	\begin{align*}
		\begin{bmatrix}
    		-1 &  4 & 8 & 2     \\
		    0  &  -1 & 5 & 6   \\
		    6 & 12 & -1 & 3\\
		    5 & 17 & 2 & -1
		\end{bmatrix}
		\text{ and }
		\begin{bmatrix}
		    -1  &  4 & 5 & 6      \\
		    0  &  -1   & 2 & 10  \\
		    5 & 17 & -1  & 3 \\
		    4 & 22 & 2 & -1
		\end{bmatrix}. 
	\end{align*}
	Then $I_S$ is minimally generated by $\RF$-relations as $I_S = \langle \phi_1,\ldots, \phi_5 \rangle,$ where
	\begin{align*}
		\phi_1:= x_2^5x_3^3-x_1x_4^4,~\phi_2:= x_1^6x_2^{13}-x_3^6x_4^3,~\phi_3:= x_1^5x_2^{18}-x_3^3x_4^7,~\phi_4:= x_1^4x_2^{23}-x_4^{11},~\phi_5:= x_3^9-x_1^7x_2^8x_4
	\end{align*}
	We see that $\phi_1,\phi_2,\phi_3,\phi_4$ are $\RF(77,58)$-relations and $\phi_5$ is a $\RF(64,69)$-relation.
\end{Example}

The following result provides a sufficient condition for the defining ideal $I_S$ to be generated
by $\RF$-relations.

\begin{Theorem} \label{RFRelation}
	Suppose $I_S$ admits a system of binomial generators $\phi_1,\ldots,\phi_m$ satisfying the following property: for each $k$ there exist $i < j$ and $f \in \PF(S)$ such that $\deg \phi_k = f + a_i + a_j$ and $\phi_k = u-v$, where $u$ and $v$ are monomials such that $x_i|u$ and $x_j|v$. Then $I_S$ is generated by $\RF$-relations.
\end{Theorem}

\begin{proof}
	Same as proof of \cite[Lemma 3.7]{herzogWatanabe}.
\end{proof}

\begin{Remark} \label{RFGluing}
	Let $S = \langle a_1,\ldots,a_n \rangle$ be a $\MPD$-semigroup. Suppose $S = S_1 +_d S_2,$ where $S_1 = \langle a_1,\ldots,a_e \rangle$ and $S_2 = \langle a_{e+1},\ldots,a_n \rangle.$ Let $h \in \PF(S).$ Then by Theorem {\rm\ref{PFGluing}}, there exist $f \in \PF(S_1)$ and $g \in \PF(S_2)$ such that $h = f+g+d.$ Since $f \in \PF(S_1),$ $f+d \in S_1.$ Write $f+d = \sum_{j=1}^e m_j a_j.$ Similarly, as $g+d \in S_2$, write $g+d = \sum_{j=e+1}^n m_j a_j.$ Hence the matrix
	\begin{align*}
		\left[
		\begin{array}{c|c} 
			\RF(f) & B \\ 
			\hline 
			C & \RF(g) 
		\end{array} 
		\right],
	\end{align*}
	where each row of the matrix $B$ is $(m_{e+1},\ldots,m_n)$ and each row of matrix $C$ is $(m_1,\ldots,m_e)$, serves as an $\RF$-matrix for $h.$
\end{Remark}

\begin{Theorem}  \label{GenRFRelations}
	Let $S = S_1 +_d S_2$ be a gluing of $\MPD$-semigroups $S_1 = \langle a_1,\ldots,a_e \rangle$ and $S_2= \langle a_1',\ldots,a_{e'}' \rangle$ such that $I_{S_1}$ and $I_{S_2}$ are generated by $\RF$-relations. If there exist $f \in \PF(S_1),$ $g \in \PF(S_2)$ with $f+d = \sum_{j=1}^e m_ja_j$, $g+d= \sum_{j=1}^{e'}m_j'a_j'$ and $(m^{(1)},\ldots,m^{(e)}), (m'^{(1)},\ldots,m'^{(e')})$ rows of $\RF(f)$, $\RF(g)$ respectively such that

	$(m_j - m^{(j)})$ and $ (m_j' - m'^{(j)})$ are non negative for all $j.$ Then $I_S$ is generated by $\RF$-relations.
\end{Theorem}

\begin{proof}
	Let the sets $\rho_1$ and $\rho_2$ be the system of generators of $I_{S_1}$ and $I_{S_2}$ in $k[x_1,\ldots,x_e]$ and $k[y_1,\ldots,y_{e'}]$ respectively. Since $d \in S_1 \cap S_2$ and 
	\[d = \sum_{j=1}^e (m_j - m^{(j)})a_j = \sum_{j=1}^{e'} (m_j' - m'^{(j)})a_j', \] 
	we can define $\rho = \prod_{j=1}^e x_j^{(m_j-m^{(j)})} - \prod_{j=1}^{e'} y_j^{(m_j'-m'^{(j)})} \in I_S.$ Using \cite[Theorem 1.4]{rosales97}, it follows that the set $\sigma = \rho_1 \cup \rho_2 \cup \{\rho\}$ is a system of generators for $I_S$ in $k[x_1,\ldots,x_e,y_1,\ldots,y_{e'}].$ It is therefore sufficient to show that for each $\phi \in \sigma,$ there is an $h \in \PF(S)$ such that $\phi$ is an $\RF(h)$-relation.\\
	
	\textbf{Case-1}: Let $\phi \in \rho_1 \cup \rho_2.$ If $\phi \in \rho_1,$ then as $I_{S_1}$ is generated by $\RF$-relations, there exists an $f' \in \PF(S_1)$ such that $\phi$ is an $\RF(f')$-relation. Using Theorem \ref{PFGluing}, we get $f' + g + d \in \PF(S)$ and from Remark \ref{RFGluing}, it follows that the following matrix serves as an $\RF$-matrix for $f' + g + d,$
	\begin{align*}
	\RF(f' + g + d) =
	\left[
	\begin{array}{c|c} 
  		\RF(f') & B \\ 
  		\hline 
  		C & \RF(g) 
	\end{array} 
	\right],
	\end{align*}
	where each row of $B$ is $(m'_1,\ldots,m'_{e'})$ and each row of $C$ is  $(k_1,\ldots,k_e)$  such that $  \sum_{j=1}^{e}k_j a_j$ is equal to $f' + d $. Since each row of $B$ is equal, we get $\phi$ is an $\RF(f' + g + d)$-relation. Now let $\phi \in \rho_2.$ Since $I_{S_2}$ is generated by $\RF$-relations, there exists a $g' \in \PF(S_2)$ such that $\phi$ is an $\RF(g')$-relation and consequently, $\phi$ is an $\RF(f + g' + d)$-relation.\\


	\textbf{Case-2}: Let $\phi = \rho.$ Since $g \in \PF(S_2)$ and $f \in \PF(S_1),$ we have  $f + g + d \in \PF(S)$ and that the following matrix serves as an $\RF$-matrix for $f+g+d,$
	\begin{align*}
	\RF(f + g + d)=
	\left[
	\begin{array}{c|c} 
  		\RF(f) & B \\ 
  		\hline 
  		C & \RF(g) 
	\end{array} 
	\right],
	\end{align*}
	where each row of $B$ is $(m'_1,\ldots,m'_{e'})$ and each row of $C$ is $(m_1,\ldots,m_e).$ Using the assumption that there exists rows $(m^{(1)},\ldots,m^{(e)})$ and $(m'^{(1)},\ldots,m'^{(e')})$ of $\RF(f)$ and $\RF(g)$ respectively such that $(m_j - m^{(j)})$ and $ (m_j' - m'^{(j)})$ are non negative for all $j,$ it follows that $\phi = \rho$ is an $\RF(f+g+d)$-relation.
	
	Hence $I_S$ is generated by $\RF$-relations.
\end{proof}

\begin{Example}
	Let $S$ be the semigroup generated by
	\[ \{(0,9), (18, 0), (27, 0), (9, 18), (8, 8), (10, 10)\}. \] 
	Then $S$ is not a $\Cc$-semigroup and $S = S_1 +_d S_2$, where $d = (18,18)$, 
	\[ S_1 = \langle (0,9), (18, 0), (27, 0), (9, 18) \rangle \text{ and } S_2 = \langle (8, 8), (10, 10) \rangle.\]
We have ideals $I_{S_1} = (x_1^2x_3-x_2x_4, x_2^3-x_3^2, x_1^4x_2-x_4^2, x_1^2x_2^2-x_3x_4) \subseteq k[x_1,\ldots,x_4]$ and $I_{S_2}= (y_1^5-y_2^4) \subseteq k[y_1,y_2]$. Also, $\PF(S_1) = \{(9, 9)\}$ and $\PF(S_2)=\{(22, 22)\}$. Note that $I_{S_1}$ and $I_{S_2}$ are generated by $\RF$-relations. Set $f=(9,9)$ and $g=(22,22).$ Then $\RF(f)$ and $\RF(g)$ are 
	\begin{align*}
		\left[
		\begin{array}{rrrr}
			-1 & 0 & 0 & 1 \\
			1 & -1 & 1 & 0 \\
			1 & 2 & -1 & 0 \\
			3 & 1 & 0 & -1 
		\end{array}
		\right]
		\quad \text{ and } \quad 
		\left[
		\begin{array}{rr}
			-1 & 3 \\
			4 & -1 
		\end{array}
		\right]
		\text{ respectively.}
	\end{align*}   

Since $f+d = 3 \cdot (0,9) + 0 \cdot (18, 0) + 1 \cdot (27, 0) + 0 \cdot (9, 18)$ and $g + d = 0 \cdot (8, 8) + 4 \cdot (10, 10)$, we pick the rows $(m^{(1)}, m^{(2)}, m^{(3)}, m^{(4)}) = (1,-1,1,0)$ and $(m'^{(1)},m'^{(2)}) = (-1,3)$ of $\RF(f)$ and $\RF(g)$ respectively. Using Theorem {\rm \ref{GenRFRelations}}, it then follows that 
\[
I_S = (x_1^2x_3-x_2x_4, x_2^3-x_3^2, x_1^4x_2-x_4^2, x_1^2x_2^2-x_3x_4,y_1^5-y_2^4, x_1^2x_2-y_1y_2) \subseteq k[x_1,x_2,x_3,x_4,y_1,y_2],
\]
is generated by $\RF$-relations.
\end{Example}

\section{$\RF$-matrices and generic toric ideals}

Let $S = \langle a_1,\ldots,a_n \rangle \subseteq \mathbb{N}^d$ be an affine semigroup and $I_S \subset k[x_1,\ldots,x_n]$ be the defining ideal of the semigroup ring $k[S].$ For a given vector $a  = (a_1,\ldots,a_d) \in \mathbb{N}^d$, the support of $a$ is defined as
\[
\mathrm{supp}(a) = \{ i \mid i \in [1,d],\ a_i \neq 0\}.
\]
For a monomial ${\bf x}^u$, define $\mathrm{supp} ({\bf x}^u) =\mathrm{supp}(u)$ and for a binomial ${\bf x}^u - {\bf x}^v$, define $\mathrm{supp} ({\bf x}^u - {\bf x}^v) =\mathrm{supp}(u) \cup \mathrm{supp}(v)$. In \citep{peeva-sturmfels}, Peeva and Sturmfels defined that a toric ideal $I_S \subset k[x_1,\ldots,x_n]$ is called \textbf{generic} if it is minimally generated by the binomials of full support. A binomial ${\bf x}^u - {\bf x}^v$ is called indispensable if every system of binomial generators of $I_S$ contains ${\bf x}^u - {\bf x}^v$ or ${\bf x}^v - {\bf x}^u.$  Using \citep[Theorem 3.1]{thomaindispensableArxiv} it follows that if $I_S$ is generic toric ideal, then it has a unique minimal system of generators $B(I_S)$ of indispensable binomials up to the sign of binomials. 

If $x = \sum_{j=1}^n m_ja_j$ is the unique expression for $x$ in $S,$ then we say $x$ has unique factorization in $S.$ In other words, if $x = \sum_{j=1}^n m_ja_j = \sum_{j=1}^n m_j'a_j$ are two factorizations of $x$ in $S,$ then $m_j = m_j'$ for all $j \in [1,n].$ We denote the set of such elements by $\UF(S)$.
We define a partial order $\leq_S$ on $S$ as $ x \leq_S y$ if and only if $y - x \in S$. We denote the set of minimal elements of $S \setminus \UF(S)$ with respect to $\leq_S$ by $\min(S \setminus \UF(S))$.

\begin{Definition}\label{Aperyset}
	Let $S = \langle a_1,\ldots,a_n \rangle$ be an affine semigroup and $0 \neq x \in S.$ We define a subset $\Ap(S, x)$ of $S$ relative to $x$, as follows
	\begin{align*}
		\Ap(S, x) = \{ y \in S \mid y-x \in \Hc(S) \}.
	\end{align*}
\end{Definition}

\begin{Definition}\label{Support}
	Let $S = \langle a_1,\ldots,a_n \rangle$ be an affine semigroup and $0 \neq x \in S.$ Define $support$ of $x$ as
	\begin{align*}
		\Supp(x) = \{ j' \mid \text{ there is a factorization } \sum_{j=1}^n m_ja_j \text{ of } x \text{ such that } m_{j'} > 0 \}.
	\end{align*}
\end{Definition}

\begin{Proposition}\label{UniqueRFnotUF}
	Let $S = \langle a_1,\ldots,a_n \rangle$ be a $\MPD$-semigroup. If $ \bigcup_{j=1}^n \Ap(S,a_j) \subseteq \UF(S),$ then $\RF(f)$ is unique for all $f \in \PF(S)$.  
	Moreover, if $\Hc(S)$ is finite, then the converse is also true.
\end{Proposition}
\begin{proof}
	 Let $\bigcup_{j=1}^n \Ap(S,a_j) \subseteq \UF(S).$ If $f \in \PF(S),$ then $f+a_j-a_j \in \Hc(S).$ Therefore, $f+a_j \in \Ap(S, a_j)$ for all $j \in [1,n].$ Hence, $f+a_j$ has unique factorization in $S$ for all $j \in [1,n]$ and $\RF(f)$ is unique. \\
	 Conversely, assume that $\Hc(S)$ is finite and $\RF(f)$ is unique for all $f \in \PF(S).$ Let $x \in \bigcup_{j=1}^n \Ap(S,a_j).$ Then there exists $j \in [1,n]$ such that $x \in \Ap(S, a_j).$ Therefore, $x - a_j \in \Hc(S).$ By Corollary \ref{MaxPseudoFrob}, there exists $f \in \PF(S)$ such that $ f+a_j - x \in S.$  Since $\RF(f)$ is unique, $f+a_j$ has unique factorization in $S.$ Hence $x$ has unique factorization in $S.$
\end{proof}

\begin{Proposition}\label{RFrelationimplyMinGen}
	Let $S$ be a $\MPD$-semigroup and $I_S$ be generic toric ideal. Suppose that $\phi \in I_S$ is an $\RF(f)$-relation for some $f \in \PF(S).$ Then $\phi \in B(I_S)$.
\end{Proposition}
\begin{proof}
	Let $m_i$ and $m_j$ be the rows of $\RF(f)$ for some $f \in \PF(S)$ such that $\phi = \phi_{(ij)} = {\bf{x}}^{m_{(ij)}^+} - {\bf{x}}^{m_{(ij)}^-}$. Let $d$ be the $S$-degree of ${\bf{x}}^{m_{(ij)}^+}$ and ${\bf{x}}^{m_{(ij)}^-}$. Therefore, we have $d \in S \setminus \UF(S)$. To prove $\phi \in B(I_S)$, it is sufficient to prove that $ d \in \min(S \setminus \UF(S))$. Suppose $ d \notin  \min(S \setminus \UF(S)).$ Then there exist $d' \in \min(S \setminus \UF(S))$ such that $d - d' \in S$. Write $d = d' + s$ for some $s \in S$. Now using the similar arguments as in the proof of \cite[Theorem 3.1]{thomaindispensableArxiv}\footnote{Another proof of this theorem can be found in \cite[Remark 4.4(3)]{peeva-sturmfels}. We cite Theorem 3.1 of reference \citep{thomaindispensableArxiv} because we have adopted the techniques of this proof in our article.}, we get a contradiction.
\end{proof}

\begin{Lemma}\label{minnotUF}
	If for each $x \in \min(S \setminus \UF(S))$, $\mathrm{Supp}(x) = [1,n]$ then $S \setminus \UF(S) \subseteq S \setminus \bigcup_{j=1}^n \Ap(S, a_j)$.
\end{Lemma}
\begin{proof}
	Let $y \in S \setminus \UF(S).$ Then there exists $x \in \min(S \setminus \UF(S))$ such that $x \leq_S y$. Therefore, $y = x + z$ for some $z \in S.$ We get, $y - a_j = x - a_j + z$ for all $j \in [1,n]$. Since $\mathrm{Supp}(x) = [1,n]$, $y - a_j \in S$ for all $j.$ Hence, $y \notin \Ap(S, a_j)$ for all $j$.
\end{proof}

\begin{Theorem}\label{genericRFunique}
	Let $S$ be a $\MPD$-semigroup. If $I_S$ is generic, then $\RF(f) = (m_{ij})$ is unique for each $f \in \PF(S)$ and $m_{ij} \neq m_{i'j}$ for all $i \neq i'$.
\end{Theorem}
\begin{proof}
	Since $I_S$ is generic, by \cite[Theorem 3.1]{thomaindispensableArxiv} $I_S$ has a unique set of minimal generators, generated by indispensable binomials. We claim that $s \in S$ is equal to degree of $\phi$ for some indispensable binomial $\phi$ in the minimal generating set of $I_S$ if and only if $s \in \min(S \setminus \UF(S))$. If $s \in S$ is equal to degree of $\phi$ for some indispensable binomial $\phi$ in the minimal generating set of $I_S$, then $s$ does not have a unique factorization and using the arguments as in the proof of \cite[Theorem 3.1]{thomaindispensableArxiv}, we get $s \in \min(S \setminus \UF(S))$. Conversely, let $s \in \min(S \setminus \UF(S)).$ Suppose that $s$ has more than two different factorizations. Let ${\bf x}^u$, ${\bf x}^v$ and ${\bf x}^w$ denote the monomials corresponding to the different factorizations of $s.$ Then $\deg {\bf x}^u = \deg {\bf x}^v = \deg {\bf x}^w = s.$ Using minimality of $s$, without loss of generality, we may assume that ${\bf x}^u - {\bf x}^v$ is in the minimal generating set of $I_S.$ Since $I_S$ is generic, it follows that if ${\bf x}^u = \prod_{i \in \Lambda} x_i^{\alpha_i}$ for some $\Lambda \subsetneq [1,n]$ and $\alpha_i \in \mathbb{N}$, then ${\bf x}^v = \prod_{i \in [1,n] \setminus \Lambda} x_i^{\alpha_i}.$ Using \cite[Proposition 2.4]{thomaindispensable}, we get $\gcd({\bf x}^u,{\bf x}^w)=1$ and $\gcd({\bf x}^v, {\bf x}^w)=1.$ Note that $\gcd({\bf x}^u,{\bf x}^w)=1$ implies that there exists $j \in [1,n] \setminus \Lambda$ such that $x_j \mid {\bf x}^w$ which contradicts that $\gcd({\bf x}^v,{\bf x}^w)=1.$ Thus, $s$ has precisely two different factorizations. Using the minimality of $s$, the claim holds. Now, let $s \in \min(S \setminus \UF(S))$ and since $I_S$ is generic, we have $\mathrm{Supp}(s) = [1,n]$. By Lemma \ref{minnotUF}, we have $S \setminus \UF(S) \subseteq S \setminus \bigcup_{j=1}^n \Ap(S, a_j).$ Hence, by Proposition \ref{UniqueRFnotUF}, we have $\RF(f) = (m_{ij})$ is unique for each $f \in \PF(S)$. Now, let $f \in \PF(S)$ and $m_i ,m_{i'}$ be two different rows of $\RF(f)$. Set $ m _{(ii')} = m_i - m_{i'}.$ Then by Proposition \ref{RFrelationimplyMinGen}, we get $ \phi_{(ii')} = {\bf{x}}^{m_{(ii')}^+} - {\bf{x}}^{m_{(ii')}^-} \in B(I_S)$. Therefore $\phi_{(ii')}$ is a binomial of full support and hence $m_{ij} \neq m_{i'j}$ for all $i \neq i'$.
\end{proof}

\begin{Example}
	Let $S = \langle (20,0), (24,1), (1,25), (0,31) \rangle.$ Then
	\begin{align*}
		\PF(S) =
		\left\lbrace
		\begin{array}{c}
			(223,4445), (271,3145), (319,1845), (559,1256), (799,667), \\ 
			(1375,567), (1951,467), (2527,367), (3103,267)
		\end{array}
		\right\rbrace	 
	\end{align*} 
	and $I_S$ is generated by 
\begin{align*}
	\left\lbrace
	\begin{array}{c}
	x_2^{24}x_3^4-x_1^{29}x_4^4, \ x_1^{12}x_3^{24}-x_2^{11}x_4^{19}, \ x_2^{13}x_3^{28}-x_1^{17}x_4^{23}, \ x_1^{41}x_3^{20}-x_2^{35}x_4^{15}, \ x_2^2x_3^{52}-x_1^5x_4^{42}, \ x_1^{70}x_3^{16}-x_2^{59}x_4^{11}, \\[2mm] x_1^7x_3^{76}-x_2^9x_4^{61}, \ x_1^{99}x_3^{12}-x_2^{83}x_4^7, \ x_1^{128}x_3^8-x_2^{107}x_4^3, \ x_2^{131}-x_1^{157}x_3^4x_4, \ x_1^2x_3^{128}-x_2^7x_4^{103}, \ x_3^{180}-x_1^3x_2^5x_4^{145}
	\end{array}
	\right\rbrace.
\end{align*}
Hence, $I_S$ is generic and $\RF$-matrices for the elements of $\PF(S)$ are
\begin{align*}
		\begin{bmatrix}
    		-1 &  8 & 51 & 102     \\
		    6  &  -1 & 127 & 41   \\
		    4 & 6 & -1 & 144\\
		    1 & 1 & 179 & -1
		\end{bmatrix},
		\begin{bmatrix}
		    -1  &  10 & 51 & 60      \\
		    11  &  -1   & 75 & 41  \\
		    4 & 8 & -1  & 102 \\
		    6 & 1 & 127 & -1
		\end{bmatrix},
		 \begin{bmatrix}
    		-1 &  12 & 51 & 18     \\
		    16  &  -1 & 23 & 41   \\
		    4 & 10 & -1 & 60\\
		    11 & 1 & 75 & -1
		\end{bmatrix},
		\begin{bmatrix}
    		-1 &  23 & 27 & 18     \\
		    28  &  -1 & 23 & 22   \\
		    16 & 10 & -1 & 41\\
		    11 & 12 & 51 & -1
		\end{bmatrix},
		\begin{bmatrix}
    		-1 &  34 & 3 & 18     \\
		    40  &  -1 & 23 & 3  \\
		    28 & 10 & -1 & 22\\
		    11 & 23 & 27 & -1
		\end{bmatrix},
	\end{align*} 
	
	\begin{align*}
		\begin{bmatrix}
    		-1 &  58 & 3 & 14     \\
		    69  &  -1 & 19 & 3   \\
		    28 & 34 & -1 & 18\\
		    40 & 23 & 23 & -1
		\end{bmatrix},
		\begin{bmatrix}
		    -1  &  82 & 3 & 10      \\
		    98  &  -1   & 15 & 3  \\
		    28 & 58 & -1  & 14 \\
		    69 & 23 & 19 & -1
		\end{bmatrix},
		 \begin{bmatrix}
    		-1 &  106 & 3 & 6     \\
		    127  &  -1 & 11 & 3   \\
		    28 & 82 & -1 & 10\\
		    98 & 23 & 15 & -1
		\end{bmatrix}
		\text{ and }
		\begin{bmatrix}
    		-1 &  130 & 3 & 2     \\
		    156  &  -1 & 7 & 3   \\
		    28 & 106 & -1 & 6\\
		    127 & 23 & 11 & -1
		\end{bmatrix} \text{respectively}.
	\end{align*} 
Moreover, these matrices are unique and no two entries in a column of each matrix are same.	
\end{Example}

\begin{Corollary}\label{GluingNotGeneric}
	Let $n \geq 3$ and $S = \langle a_1,\ldots,a_n \rangle$ be a gluing of $\MPD$-semigroups. Then $I_S$ is not generic.
\end{Corollary}
\begin{proof}
	Follows from Theorem \ref{PFGluing}, Remark \ref{RFGluing} and Theorem \ref{genericRFunique}.
\end{proof}

\begin{Theorem}
	Let $n \geq 4$ and $S = \langle a_1, \ldots , a_n \rangle$ be a $\prec$-almost symmetric $\MPD$-semigroup. Then $I_S$ is not generic.
\end{Theorem}

\begin{proof}
	Suppose that $I_S$ is generic. Since $S$ is $\prec$-almost symmetric, there exist $f, f' \in \PF(S)$ such that $f+f' = \F(S)_{\substack{\\\prec}}.$ Let $f+a_i = \sum_{j=1,j\neq i}^n m_ja_j$ be a factorization of $f+a_i$ in $S$. Suppose there exist $k \neq k'$ such that $m_k, m_{k'} \neq 0$ in the factorization of $f+a_i$. Since $f' \in \PF(S)$, we have $k \notin \mathrm{Supp}(f'+a_k)$ and $k' \notin \mathrm{Supp}(f' + a_{k'})$. Let $f' + a_k = \sum_{j=1,j \neq k}^n m'_ja_j$ and $f'+a_{k'} = \sum_{j=1,j \neq k'}^n m''_ja_j$. Since $\F(S)_{\substack{\\\prec}}+a_i = f+a_i-a_k+f'+a_k = f+a_i-a_{k'}+f'+a_{k'}$, we have $\F(S)_{\substack{\\\prec}}+a_i \notin \UF(S)$ and hence $\RF(\F(S)_{\substack{\\\prec}})$ is not unique. Which is a contradiction to Theorem \ref{genericRFunique}. Therefore, for each $i$, $f+a_i = m_ja_j$ for some $j \neq i$ and hence each row of $\RF(f) = (m_{ij})$ has at least $n-2$ zeros. Since $n \geq 4$, there exist $i \neq i'$ such that $m_{ij} = m_{i'j}$ for some $j$. Which is again a contradiction to the Theorem \ref{genericRFunique}. Hence, $I_S$ is not generic.
\end{proof}

\textbf{ACKNOWLEDGEMENT.} Experiments with the computer algebra softwares Macaulay2 \cite{M2} and GAP \cite{GAP4} have provided numerous valuable insights.

\bibliographystyle{plain}

\end{document}